\RequirePackage{fix-cm}
\documentclass[smallextended]{svjour3}       
\smartqed  
\usepackage{braket,amsfonts}

\usepackage{array}

\usepackage[caption=false]{subfig}

\usepackage{algorithmic}

\usepackage{graphicx,epstopdf}

\usepackage{amsopn}

\usepackage{xspace}
\usepackage{bold-extra}
\usepackage[most]{tcolorbox}

\colorlet{texcscolor}{blue!50!black}
\colorlet{texemcolor}{red!70!black}
\colorlet{texpreamble}{red!70!black}
\colorlet{codebackground}{black!25!white!25}

\lstdefinestyle{siamlatex}{%
  style=tcblatex,
  texcsstyle=*\color{texcscolor},
  texcsstyle=[2]\color{texemcolor},
  keywordstyle=[2]\color{texemcolor},
  moretexcs={cref,Cref,maketitle,mathcal,text,headers,email,url},
}

\tcbset{%
  colframe=black!75!white!75,
  coltitle=white,
  colback=codebackground, 
  colbacklower=white, 
  fonttitle=\bfseries,
  arc=0pt,outer arc=0pt,
  top=1pt,bottom=1pt,left=1mm,right=1mm,middle=1mm,boxsep=1mm,
  leftrule=0.3mm,rightrule=0.3mm,toprule=0.3mm,bottomrule=0.3mm,
  listing options={style=siamlatex}
}

\newtcbinputlisting[use counter=example]{\examplefile}[3][]{%
  title={Example~\thetcbcounter: #2},listing file={#3},#1}

\DeclareTotalTCBox{\code}{ v O{} }
{ 
  fontupper=\ttfamily\color{black},
  nobeforeafter,
  tcbox raise base,
  colback=codebackground,colframe=white,
  top=0pt,bottom=0pt,left=0mm,right=0mm,
  leftrule=0pt,rightrule=0pt,toprule=0mm,bottomrule=0mm,
  boxsep=0.5mm,
  #2}{#1}

\usepackage{booktabs}
\usepackage{color}
\usepackage{amsmath}
\usepackage{amssymb}
\allowdisplaybreaks
\usepackage{algorithm}

\newtheorem{assumption}{Assumption}

\usepackage{endnotes}
\let\footnote=\endnote

\usepackage{bm}

\newcommand*{\xbf}{{\mathbf x}}
\newcommand*{\x}{{\mathbf x}}
\newcommand*{\ubf}{{\mathbf u}}
\newcommand*{\ybf}{{\mathbf y}}
\newcommand*{\y}{{\mathbf y}}
\newcommand*{\sgn}{{\text{sgn}}}
\newcommand*{\zbf}{{\mathbf z}}

\newcommand*{\vbf}{{\mathbf v}}

\newcommand*{\z}{{\mathbf z}}

\newcommand*{\R}{{\mathbb R}}
\newcommand*{\bR}{{\mathbb R}}

\newcommand*{\E}{{\mathbb E}}
\newcommand*{\bE}{{\mathbb E}}

\usepackage{ulem}
\newcommand*{\lf}{{\underline f}}
\def\be{\begin{enumerate}}
\def\ee{\end{enumerate}}
\newcommand*{\eh}{{\hat{\epsilon}}}
\newcommand{\sigmainf}{\sigma_{\infty}}
\newcommand{\hrho}{\hat{\rho}}
\usepackage{dsfont}

\def\sO{\mathcal{O}}
\def\hx{\mathbf{\hat x}}
\def\prox{\rm prox}
\def\sN{\mathcal{N}}
\def\hrho{\hat \rho}
\def\conv{\rm conv}
\def\prox{\rm prox}

\def\tf{\tilde{f}}
\def\tg{\tilde{g}}
\begin{document}

\title{Stochastic First-Order Methods with Non-smooth and
Non-Euclidean Proximal Terms for Nonconvex
High-Dimensional Stochastic Optimization
\thanks{Research is supported from HKU-IDS start-up fund; and Guangdong Province Fundamental and Applied Fundamental Research Regional Joint Fund, 2022B1515130009.}
}


\author{Yue Xie     \and Jiawen Bi    \and
        Hongcheng Liu 
}


\institute{Yue Xie \at
              Department of Mathematics and Musketeers Foundation Institute of Data Science, The University of Hong Kong, Pokfulam, Hong Kong.\\
              \email{yxie21@hku.hk}           
           \and
           Jiawen Bi \at
              Department of Mathematics, The University of Hong Kong, Pokfulam, Hong Kong.\\
              \email{jwbi@hku.hk}
           \and
           Hongcheng Liu \at
              Department of Industrial and Systems Engineering, University of Florida \\
              \email{hliu@ise.ufl.edu} 
}

\date{Received: date / Accepted: date}

\maketitle

\begin{abstract}
When the nonconvex problem is complicated by stochasticity, the sample complexity of stochastic first-order methods may depend linearly on the problem dimension, which is undesirable for large-scale problems. In this work, we propose dimension-insensitive stochastic first-order methods (DISFOMs) to address nonconvex optimization with expected-valued objective function. Our algorithms allow for non-Euclidean and non-smooth distance functions as the proximal terms. Under mild assumptions, we show that DISFOM using minibatches to estimate the gradient enjoys sample complexity of $ \sO ( (\log d) / \epsilon^4 ) $ to obtain an $\epsilon$-stationary point. Furthermore, we prove that DISFOM employing variance reduction can sharpen this bound to $\sO( (\log d)^{2/3}/\epsilon^{10/3} )$, which perhaps leads to the best-known sample complexity result in terms of $d$. We provide two choices of the non-smooth distance functions, both of which allow for closed-form solutions to the proximal step. Numerical experiments are conducted to illustrate the dimension insensitive property of the proposed frameworks.

\end{abstract}

\section{Introduction}\label{section: intro}
In this paper, we consider a stochastic optimization (SO) problem formulated as below:
\begin{align}\label{SP problem}
\begin{aligned}
        \min_{\x \in \bR^d} & \quad f(\xbf) = \bE[F(\xbf,\zeta)] \\
        \mbox{s.t.} & \quad \xbf \in X, 
\end{aligned}
\end{align} 
where $f: \bR^d \to \bR$ and $F: \bR^d \times \Omega \to \bR$. $(\Omega, \mathcal{F}, \mathbb{P})$ is the probability space. $X$ is a closed and convex set.
Denote by $\xbf^*\in\R^d$ an optimal solution to \eqref{SP problem}. We do not have explicit assumption on problem dimension $d$, but we are particularly interested in a large $d$. 
We assume that the expectation   $\E\left[F(\xbf,\zeta)\right]=\int_{\Omega} F(\xbf,\zeta)\,\text{d}\mathbb P(\zeta)$ is well-defined and finite-valued for every $\xbf\in\R^d$ and that  it is possible to generate independent and identically distributed realizations, $\zeta_1,\,\zeta_2,....,$ of the random vector $\zeta$. For this problem, we are particularly interested in solution algorithms when $f(\x)$ is smooth but potentially nonconvex. 

In view of stochasticity and potential nonconvexity of the problem,  we aim to compute an stationary point with small expected value of first-order optimality residual. 
\begin{definition}\label{local stationary point definition}
For any $\epsilon\geq 0$ and nonnegative residual function $r(\cdot)$, a feasible random solution $\xbf^{*}_{\epsilon}\in X$ is said to be an {\it $\epsilon$-stationary solution/point} if and only if  
\begin{align}
\bE\left[ r(\xbf^{*}_{\epsilon}) \right]\leq \epsilon.
\end{align}
\end{definition}

Many stochastic first-order methods (S-FOMs) apply to smooth and nonconvex optimization problems in the form of \eqref{SP problem}. 
These methods date back to \cite{robbins1951stochastic,chung1954stochastic,sacks1958asymptotic}, and have been much studied (e.g., by \cite{polyak1990new,polyak1992acceleration,nemirovski2009robust,lan2012optimal,ghadimi2013stochastic,rosasco2019convergence,devolder2011stochastic,ghadimi2012optimal,chambolle2018stochastic,chambolle2011first,bach2013non,metel2019simple,xu2019stochastic,nitanda2017stochastic,davis2019stochastic,li2019ssrgd,nguyen2017sarah,wang2019spiderboost,fang2018spider,pham2020proxsarah,horvath2020adaptivity,allen2018make}, among others).
In terms of the stochastic first-order oracle or sample, most existing S-FOMs incur analytical complexities that grow polynomially in dimensionality. 

Indeed, for S-FOMs in  smooth and nonconvex optimization,  (See \cite{ghadimi2013stochastic,ghadimi2016mini,pham2020proxsarah} and references therein.) Here $\sigma^2$ is defined as an upper bound on $ \E[\Vert \nabla F(\xbf,\zeta)-\nabla  f(\xbf)\Vert^2]$. In the worst case, $\sigma^2$ is linear in $d$ because
\begin{align*}
    \sigma^2 & \ge \max_{x \in X} \bE[\Vert \nabla F(x,\zeta)-\nabla  f(x) \Vert^2]  = \max_{x \in X} \sum_{i=1}^d \bE[ ( \nabla_i F(x,\zeta) - \nabla_i f(x) )^2 ],
\end{align*}
where the right-hand side is often $\sO(d)$. This then leads to an analytical complexity in terms of {stochastic first-order oracle (SFO)} as following:
\begin{equation}\label{state-of-the-art-bound}
 \mathcal O\left( \frac{d}{\epsilon^4}\right). 
\end{equation} 
More recent results such as by \cite{fang2018spider,wang2019spiderboost,zhou2020stochastic} further enhances \eqref{state-of-the-art-bound}
 into: 
 \begin{align}\label{state-of-the-art-bound epsilon}
     {\tilde{\mathcal O}\left(\frac{d}{\epsilon^2}+\frac{\sqrt{d}}{\epsilon^3}\right),}
 \end{align}
{where $\tilde{\mathcal{O}}$ may hide some logarithm terms.}

{In recent years, efforts have been made to investigate randomized first-order approaches to address nonsmooth and nonconvex optimization \cite{zhang2020complexity,davis2022gradient,kornowski2022oracle}. Since this problem class is too large and may include extremely ill-conditioned problems, finding an $\epsilon$ first-order stationary point in the usual sense (when ${\rm dist}(0, \partial f(x)) \le \epsilon$, where $\partial f$ denotes the Clarke subdifferential) proves to be hard \cite{zhang2020complexity,kornowski2022oracle}. Therefore, these works focus on finding a $(\delta,\epsilon)$-Goldstein stationary point. In \cite{zhang2020complexity}, the scenario of expected-valued objective function is discussed. The authors report a $\mathcal{O}\left(d^{\frac{3}{2}}\epsilon^{-4} \delta^{-1}\right)$ sample complexity to locate a $(\delta,\epsilon)$-Goldstein stationary point. If smoothness is assumed, then this becomes $\mathcal{O}\left(d^{\frac{3}{2}}\epsilon^{-5}\right)$.}
 
However, with the ever-increasing demand for more sophisticated and comprehensive models, (SO) problems of excessively high dimensionality are rapidly emerging.  
Based on the existing theories, the SGD and many other S-FOM alternatives  may have to perform at least hundreds of millions of iterations or stochastic first-order oracles to achieve the desired solution quality. Without fundamental innovations vis-\`a-vis high dimensionality,  the SO problem \eqref{SP problem} will become extremely difficult to solve in the near future. 

Dimension-insensitive S-FOMs have been discussed by \cite{nemirovski2009robust,lan2020first,zhang2018convergence}, especially about mirror descent algorithm (MDA) and its analysis. The authors of \cite{agarwal2012stochastic} base their algorithm on successively solving a series of $\ell_1$-regularized optimization problems using Nesterov’s dual averaging algorithm and obtain the dimension insensitive result. However, \cite{agarwal2012stochastic,nemirovski2009robust,lan2020first} focused on convex SO problems.  \cite{zhang2018convergence} extended the MDA to nonconvex settings and showed a complexity of \begin{equation}\label{result historical 2}
\mathcal O\left(\frac{\max_{\xbf\in X}\E\left[\Vert \nabla F(\xbf,\zeta)\Vert^2_\infty\right]}{\epsilon^4}\right).
\end{equation}
This bound is verifiable to be much better than \eqref{state-of-the-art-bound} and \eqref{state-of-the-art-bound epsilon} in terms of the dependence on $d$ ($\mathcal{O}(\log d)$ instead of $\mathcal{O}(d)$ because of the infinity norm) under certain mild assumptions on $\zeta$ (e.g. sub-Gaussian). {To our knowledge, this is the only dimension-insensitive algorithm available for our problem of interest from the literature\footnotemark.} \footnotetext{{In \cite{carmon2018accelerated} and \cite{fang2018spider}, dimension-insensitive complexity may also be derived, but for a strict subclass of \eqref{SP problem}. In \cite{carmon2018accelerated}, unconstrained deterministic optimization is considered where the objective function has closed-form and second-order Lipschitz smooth. In \cite{fang2018spider}, dimension-insensitive result only holds when the objective function is a finite-sum, a special case of \eqref{SP problem}. }} However, the result is lacking in two key respects: (i). The authors assume that the distance generating function should be continuously differentiable. This is somewhat restrictive and does not allow non-smooth alternatives to enhance the algorithm. (ii). The term $\epsilon^{-4}$ in \eqref{result historical 2} can be further improved.
In this work, we propose dimension-insensitive stochastic first-order methods (DISFOMs) that allow non-Euclidean and non-smooth distances to construct the proximal term. Throughout this work, the residual function in Def.~\ref{local stationary point definition} is naturally defined via the following:
\begin{align}\label{def: resfunc}
\begin{aligned}
r(\bar \x) & \triangleq {\rm dist}_{\| \cdot \|_\infty}(0, \partial(f + \delta_X)(\bar \x) ) \triangleq \inf\{ \| \vbf \|_\infty \mid \vbf \in \partial(f + \delta_X)(\bar \x) \}.
\end{aligned}
\end{align}
We point out that although the residual function is given in $\| \cdot \|_\infty$, the dimension-insensitive result is not trivial. In fact, the residual function is stronger than a residual function discussed in literature (\cite{zhang2018convergence}) and defined in $\| \cdot \|_1$:
\begin{align*}
    r_1(\x) \triangleq \frac{1}{\lambda} \| \x - {\prox}_{\lambda f}(\x) \|_1,
\end{align*}
where $ {\prox}_{\lambda f}(\x) $ is the proximal operator w.r.t. Bregman divergence. Detailed discussions are in Section~\ref{sec: res.f}. 

Under mild assumptions, we show that DISFOM using minibatches to estimate the gradient enjoys sample complexity of $ \sO \left( (\log d) / \epsilon^4 \right) $ to obtain an $\epsilon$-stationary point. Furthermore, we prove that DISFOM employing variance reduction can sharpen this bound to $\sO\left( (\log d)^{2/3}/\epsilon^{10/3} \right)$, which further improves over \eqref{result historical 2} in terms of dependence on both $d$ and $\epsilon$, and perhaps leads to the best-known sample complexity result in terms of $d$ in the nonconvex regime. We discuss two choices of the non-smooth distance functions, both of which allow closed-form solutions to the proximal step. Numerical experiments are conducted to illustrate the dimension-insensitive property of the proposed frameworks and show comparison with other popular algorithms.

In Section~\ref{sec: Prelim} we clarify some mild assumptions used in this article. We discuss the residual function defined in \eqref{def: resfunc} and how it compares with others in Section~\ref{sec: res.f}. In Section~\ref{sec: DISFOM}, we formally introduce the DISFOMs. Theoretical analysis are provided in Section~\ref{sec: analysis}. Section~\ref{section:num} includes numerical experiments.
A final conclusion can be found in Section~\ref{sec: conclusion}.

\section{Preliminary}\label{sec: Prelim}
We consider problem~\ref{SP problem}  where the following assumptions are satisfied:

\begin{assumption}\label{assumption first-order oracle}
There is a stochastic first-order oracle (SFO) that returns the gradient $\nabla F(\xbf,\zeta)$ for any given input point $(\xbf,\zeta)\in\R^d\times \Omega$.
\end{assumption}



\begin{assumption}\label{ass: subG}
    Let $w(\xbf) \triangleq \nabla F(\xbf,\zeta) - \nabla f(\xbf)$. $w_i(\xbf)$ denotes its $i$th component. Then
    \be
    \item $\bE[w(\xbf)] = 0, \quad \forall \xbf \in X $,
    \item $\bE[\exp(t w_i(\xbf))] \le \exp(\sigma_\infty^2 t^2 /2)$, $ \forall t \in \bR$, $\forall i = 1,\hdots,d$, $\forall \xbf \in X$,
    \ee
i.e., $w_i(\xbf)$ is a random variable that has $\sigma_\infty^2$-sub-Gaussian distribution.
\end{assumption}
By Assumption~\ref{ass: subG}, we immediately have 
\begin{align}
    \bE[\| w(\xbf) \|^2] \le d \sigma_\infty^2.
\end{align}
\begin{assumption}\label{ass: Lip}
    \begin{align}
    \begin{aligned}
        \| \nabla F(\xbf,\zeta) - \nabla F(\ybf,\zeta) \| \le L \| \xbf - \ybf \|, \quad \forall \zeta \in \Omega, \quad \forall \xbf, \ybf \in X.
        \end{aligned}
    \end{align}
\end{assumption}
By Assumption~\ref{ass: Lip}, we have that $\nabla f(\xbf)$ is also Lipschitz continuous with constant $L$. Therefore, the following inequality holds:
\begin{align}\label{ineq: Taylor}
    f(\ybf) \le f(\xbf) + \nabla f(\xbf)^T (\ybf - \xbf) + \frac{L}{2} \| \ybf - \xbf \|^2.
\end{align}
\begin{assumption}\label{ass: setw3hj}
\begin{align}
    f(\xbf) \ge \lf, \quad \forall \xbf \in X.
\end{align}
\end{assumption}

Assumption \ref{assumption first-order oracle} stipulates that the oracle accessible here is the tractable computation of a stochastic gradient. Assumption \ref{ass: subG} assumes that each component of the stochastic gradient is unbiased and its error is a sub-Gaussian random variable with parameter $\sigma_\infty$. Assumption \ref{ass: Lip} states that function $\nabla F(\cdot,\zeta)$ is Lipschitz smooth. All these  assumptions above are standard or equivalent to the common conditions in the literature of stochastic approximation or S-FOMs. For example, Assumptions \ref{assumption first-order oracle} and  \ref{ass: subG}  are imposed by \cite{nemirovski2009robust,ghadimi2013stochastic,lan2020first} (though some important results therein do not rely on Assumption \ref{ass: subG}). Note that if a random variable has a bounded support, it is sub-Gaussian. Assumption \ref{ass: Lip} or its slightly weaker version (inequality holds after taking expectation on the left) is necessary for the discussions by \cite{wang2019spiderboost,nguyen2017sarah,pham2020proxsarah}. 

\paragraph{Notation.} Given a sequence $\{ \vbf^k \}$, $\Delta \vbf^{k+1} \triangleq \vbf^{k+1} - \vbf^k$. Given a vector $\vbf \in \bR^d$, $\vbf_i$ means its $i$th component unless specified otherwise in the context. For a positive integer $K$, $[K]$ denotes the set $\{1,2,\hdots,K\}$. $\| \cdot \|$ denotes the $2$-norm $\| \cdot \|_2$. Denote $g^k \triangleq \nabla f(\xbf^k)$ for any $k$.

\section{Discussion of residual function}\label{sec: res.f}
It is true that the dimension-insensitive complexity result is dependent on the criterion used to measure optimality. If the criterion is weak, then the dimension-insensitive result could be meaningless. We demonstrate that the dimension-insensitive result associated with the residual function~\eqref{def: resfunc} is nontrivial by comparing $r(\x)$ in \eqref{def: resfunc} with the following two discussed in \cite{zhang2018convergence}:
\begin{align*}
    & r_1(\x) \triangleq \frac{1}{\lambda} \| \x - {\rm prox}_{\lambda f}(\x) \|_1, \\
    & r_2(\x) \triangleq \frac{1}{\lambda} \sqrt{ D_h(\x,{\prox}(\x)) + D_h({\prox}(\x),\x)},
\end{align*}
where 
\begin{align}\label{def: prox}
{\rm prox}_{\lambda f}(\x) \triangleq \mbox{arg}\min\limits_{\z \in X}{ f(\z) + \frac{1}{\lambda} D_h(\z,\x) }
\end{align}
and $D_h(\cdot,\cdot)$ is the Bregman divergence induced by 1-stongly convex continuously differentiable distantce generating function $h(\cdot)$ w.r.t $\| \cdot \|_1$, i.e.,
\[
D_h(\ubf,\vbf) \triangleq h(\ubf) - h(\vbf) - \langle \nabla h(\vbf), \ubf - \vbf \rangle \ge \frac{1}{2}\| \ubf - \vbf \|_1^2.
\]
The residual function $r_1(\x)$ is a natural generalization from the one used in \cite{davis2018stochastic} to $\ell_1$-norm. In fact, $r_2(\x)$ is stronger than $r_1(\x)$ since 
\begin{align*}
    & D_h(\x,{\prox}(\x)) + D_h({\prox}(\x),\x) = \langle \nabla h(\x) - \nabla h({\prox}(\x)), \x - {\prox}(\x) \rangle \\
    & \ge \| \x - {\prox}(\x) \|_1^2.
\end{align*}
We point out that $r(\x)$ is stronger than $r_2(\x)$. In fact, if $f$ is $\rho$-relatively weakly convex ($f(\x) + \rho h(\x)$ is convex for such $h(\cdot)$ and some $\rho > 0$) and $\lambda \in (0,\rho^{-1})$ so that \eqref{def: prox} is well-defined,  we have the following statement.
\begin{lemma}\label{lm: comp.mea}
Let $\x \in X$, then
    \begin{align}\label{ineq: rf0}
    \begin{aligned}
 r(\x) \ge (1-\lambda \rho)r_2(\x).
 \end{aligned}
\end{align}
\end{lemma}
\begin{proof}
    Denote $\hx \triangleq \prox_{\lambda f}(\x)$. Then by definition \eqref{def: prox}, 
\begin{align}
\notag
    & 0 \in \nabla f(\hx) + \frac{1}{\lambda}( \nabla h(\hx) - \nabla h(\x) ) + \sN_X(\hx), \\
    \notag
    & \Longleftrightarrow \left\langle \left( \nabla f(\hx) + \frac{1}{\lambda} ( \nabla h(\hx) - \nabla h(\x) ) \right) , \y - \hx \right\rangle \ge 0, \quad \forall \y \in X, \\
    \notag
    & \implies \langle \nabla f(\x), \x - \hx \rangle \ge \left\langle \nabla f (\hx) + \frac{1}{\lambda} \nabla h(\hx) - \nabla f(\x) - \frac{1}{\lambda} \nabla h(\x), \hx - \x \right\rangle \\
    \label{ineq: rf1}
    & \qquad \ge \left( \frac{1}{\lambda} - \rho \right) \langle \nabla h(\hx) - \nabla h(\x), \hx - \x \rangle.
\end{align}
Choose arbitary $\vbf \in \partial(f + \delta_X)(\x)$, then we have
\begin{align}\label{ineq: rf2}
    \left\langle \vbf - \nabla f(\x), \y - \x \right\rangle \le 0, \quad \forall \y \in X.
\end{align}
By \eqref{ineq: rf1} and \eqref{ineq: rf2}, we have
\begin{align}\label{ineq: rf3}
    \left( 1/\lambda - \rho \right) \langle \nabla h(\hx) - \nabla h(\x), \hx - \x \rangle \le \langle \vbf , \x - \hx \rangle.
\end{align}
WLOG, suppose that $\hx \neq \x$. Then by \eqref{ineq: rf3},
\begin{align}\label{ineq: rf4}
    \left( 1/\lambda - \rho \right) \frac{\langle \nabla h(\hx) - \nabla h(\x), \hx - \x \rangle} {\| \hx - \x \|_1} \le \left\langle \vbf , \frac{\x - \hx}{\| \x - \hx \|_1} \right\rangle \le \| \vbf \|_\infty.
\end{align}
However, 
\begin{align}
\notag
    & \langle \nabla h(\hx) - \nabla h(\x), \hx - \x \rangle \ge \| \hx - \x \|_1^2 \\
    \label{ineq: rf5}
   \implies &  \frac{\langle \nabla h(\hx) - \nabla h(\x), \hx - \x \rangle} {\| \hx - \x \|_1} \ge \sqrt{\langle \nabla h(\hx) - \nabla h(\x), \hx - \x \rangle}.
\end{align}
By combining \eqref{ineq: rf4}\eqref{ineq: rf5} and the fact that $\vbf$ is arbitrary, we have that \eqref{ineq: rf0} holds. \qed
\end{proof}

\section{Dimension-insensitive stochastic first-order methods}\label{sec: DISFOM}

We start from a general framework of a stochastic first-order method as follows.

\begin{algorithm}[ht] 
\caption{Stochastic first-order method}\label{alg: DI-SGD}
\begin{description}
\item[{\bf Step 1.}] Initialize $\xbf^1$ and hyper-parameters $K$, $\eta > 0$.
\item[{\bf Step 2.}] Invoke the following operations for $k=1,...,K$:
\begin{description}
 \item[{\bf Step 2.1}] Generate the gradient estimate $G^k$. 
 \item[{\bf Step 2.2}] Solve the proximal projection problem:
 \begin{align}\label{proj}
     \xbf^{k+1} = P_X^k (\xbf^k - \eta G^k)
 \end{align}
\end{description}
\item[{\bf Step 3.}] Output $\xbf^{Y+1}$ for a random $Y$, which has a discrete distribution on $[K]$ with a probability mass function $\mathbb P[Y=k]=1/K$.
\end{description}
\end{algorithm}

Now we discuss the details of the proximal projection operator $P_X^k(\cdot)$ and gradient estimate $G^k$ that can achieve the dimension-insensitive property. In particular, the proximal projection operator $P_X^k(\cdot)$ is defined as
\begin{align}\label{def: proj}
    P_X^k(\vbf) = \mbox{arg}\min_{\xbf \in X} \left\{ \frac{1}{2} \| \xbf - \vbf \|^2 + \phi(\xbf - \x
^k) \right\},
\end{align}
where $\phi: \bR^d \to \bR$ is a convex function and we suppose that calculation of \eqref{def: proj} is easy and exact when $X = \bR^d$ (we will discuss the specific form of $\phi$ later) or when $\phi \equiv 0$. Moreover, suppose that $\phi(0) = 0$. We could employ ADMM (see Appendix~\ref{app: alg}) to resolve \eqref{proj}\eqref{def: proj} to accuracy level $\eh > 0$ so that the following conditions hold: $\exists \xi^{k+1}, \Gamma^{k+1}, \ybf^{k+1}$, s.t.,
\begin{subequations}\label{ineq: proj}
\begin{align}\label{ineq: proj1} 
        (\Delta \xbf^{k+1} + \eta G^k + \xi^{k+1} + \Gamma^{k+1})^T( \xbf - \xbf^{k+1} ) \ge 0, &\quad 
        \forall \xbf \in X, \\
        \label{eq: proj2}
        \xi^{k+1} \in \partial \phi(\ybf^{k+1} - \x^k), & \\
        \label{ineq: proj3}
        {\| \Gamma^{k+1} \| \le \eh}, \quad \| \xbf^{k+1} -\ybf^{k+1} \|_1 \le \eh. &
    \end{align}
\end{subequations}
In particular, when $X = \bR^d$, by assumption we have that \eqref{def: proj} can be solved easily and exactly, and \eqref{ineq: proj1}-\eqref{ineq: proj3} hold true with $\eh = 0$. In fact, when $X = \bR^d$ we have
\begin{align} \label{eq: proj4}
    - \Delta \x^{k+1} - \eta G^k = \xi^{k+1} \in \partial \phi(\x^{k+1} - \x^k).
\end{align}
If $X \neq \bR^d$, we point out that ADMM is also fast in resolving \eqref{def: proj} under the assumption that $X$ is a polyhedron and $\partial \phi(\x)$ is a piecewise linear multifunctions ($Gr(\partial \phi(\x)) \triangleq \{ (\x,\y) \mid \y \in \partial \phi(\x) \}$ is the union of finitely many polyhedra). According to \cite{yang2016linear}, ADMM converges in {\bf linear rate} in 2-norm in this scenario. Considering that each step of ADMM is cheap, additional computation introduced by $P_X^k(\cdot)$ could be incremental. {Details about how to apply ADMM with linear convergence rate to solve \eqref{proj}\eqref{def: proj} and guarantee \eqref{ineq: proj1}-\eqref{ineq: proj3} is included in Appendix~\ref{app: alg}.}
\subsection{Choices of $\phi$}
Next we provide two possible choices of $\phi$. \\
\noindent{Case 1}: Let 
\begin{align}\label{setting1}
\phi(\x) = \frac{\hrho}{2} \| \x \|_1^2, \quad \hrho > 0.
\end{align}
For such a $\phi(\x)$, the following lemma holds by application of Danskin's Theorem.
\begin{lemma}\label{lm: l1subdiff}
Let $\phi(\x)$ be in \eqref{setting1}. 
Then
\begin{align}\label{l1subdiff}
    \partial \phi(\x) & = \hrho \| \x \|_1 \partial \| \x \|_1 = \left\{ \hrho \| \x \|_1 \vbf: \vbf \in \bR^d, \begin{array}{ll}
         \vbf_i = 1 & \mbox{if } \x_i > 0 \\
         \vbf_i = -1 & \mbox{if } \x_i < 0\\
         \vbf_i \in [-1,1] & \mbox{if } \x_i = 0
    \end{array} \right\}.
\end{align}
\end{lemma}
\begin{proof}
    Let $\conv(S)$ denote the convex hull of a set $S$. Note that $\| \x \|_1^2 = \max_{\alpha \in [-1,1]^d} (\alpha^T \x)^2 $. Then by Danskin's Theorem,
    {\small
    \begin{align*}
        & \partial \| \x \|_1^2 \\
        & = \conv\left\{ 2(\alpha^T \x)\alpha : \begin{cases}
            \alpha_i = 1 & \mbox{if } \x_i > 0, \\
            \alpha_i = -1 & \mbox{if } \x_i < 0, \\
            \alpha_i \in [-1,1] & \mbox{if } \x_i = 0,
        \end{cases} \mbox{ or }
        \begin{cases}
            \alpha_i = -1 & \mbox{if } \x_i > 0, \\
            \alpha_i = 1 & \mbox{if } \x_i < 0, \\
            \alpha_i \in [-1,1] & \mbox{if } \x_i = 0,
        \end{cases} 
        \qquad \right\}\\
        & = \conv\left\{ 2\| \x \|_1 \alpha : \begin{cases}
            \alpha_i = 1 & \mbox{if } \x_i > 0, \\
            \alpha_i = -1 & \mbox{if } \x_i < 0, \\
            \alpha_i \in [-1,1] & \mbox{if } \x_i = 0,
        \end{cases}  
        \qquad \right\} \\
        & = 2\| \x \|_1 \partial \| \x \|_1.
    \end{align*}}
    Therefore the result follows. \qed
\end{proof}
Now we discuss the formula to calculate the following proximal problem (one of the subproblems in ADMM iterations) for such a choice of $\phi$:
\begin{align}\label{opt: l1squareprox}
   \min_{\z} \quad \frac{1}{2} \| \z - \vbf \|^2 + \frac{\hat \rho}{2} \| \z \|_1^2.
\end{align}
The next lemma demonstrates that it is indeed tractable to find the solution of \eqref{opt: l1squareprox}.
\begin{lemma}\label{lm: comp.l12prox}
Suppose that $\z^*$ denotes the optimal solution of \eqref{opt: l1squareprox}. Then, if $\vbf = 0$, $\z^* = 0$. Otherwise, suppose that $\vbf \neq 0$. Let $i_1, i_2, \hdots, i_d$ be the permutation of $\{1,\hdots, d \}$ such that
\begin{align*}
    |\vbf_{i_1}| \le |\vbf_{i_2}| \le \hdots \le |\vbf_{i_d}|.
\end{align*}
Denote a dummy index $i_0 = 0$ and dummy scalar $\vbf_{i_0} = 0$. Let $s_k \triangleq \sum_{t=0}^{k-1} | \vbf_{i_t} |$ $ + \frac{(d-k+1)\hrho + 1}{\hrho} | \vbf_{i_k } |$, $\forall k = 0,...,d$. (By convention, $\sum_{t=0}^{-1} | \vbf_{i_t} | \triangleq 0$). Then $s_k$ is non-decreasing and $\| \vbf \|_1 < s_d$. Suppose that $\bar k$ satisfies $s_{\bar k} \le \| \vbf \|_1 < s_{\bar k + 1}$. Then we have that
\begin{align}\label{proxform2}
    \z_{i_t}^* = \begin{cases}
        0 & \mbox{if } 1 \le t \le \bar k \\
        \vbf_{i_t} - \frac{\sgn(\vbf_{i_t}) \hrho}{\hrho(d-k)+1} \sum_{t = k+1}^d | \vbf_{i_t} | & \mbox{if } t > \bar k.
    \end{cases}
\end{align}
\end{lemma}
\begin{proof}
When $\vbf = 0$, it is trivial to see that $\z^* = 0$. Suppose that $\vbf \neq 0$. Note that we have 
\begin{align}
\notag
  0 & \in \z^* - \vbf + \frac{\hrho}{2} \partial \| \z \|_1^2 \mid_{\z = \z^*} \overset{\eqref{l1subdiff}}{=} \z^* - \vbf + \hrho  \| \z^* \|_1 \partial \| \z \|_1 \mid_{\z = \z^*} \\
  \notag
  \Longleftrightarrow \quad \z^* & = \mbox{arg}\min_{\z} \frac{1}{2} \| \z - \vbf \|^2 + \hrho \| \z^* \|_1 \| \z \|_1 \\
  \label{proxform1}
  \Longleftrightarrow \quad \z_i^* & = \begin{cases}
      0 & \mbox{if } | \vbf_i | \le \hrho \| \z^* \|_1 \\
      \vbf_i - \hrho \| \z^* \|_1 & \mbox{if } \vbf_i > \hrho \| \z^* \|_1 \\ 
      \vbf_i + \hrho \| \z^* \|_1 & \mbox{if } \vbf_i < -\hrho \| \z^* \|_1
  \end{cases}
\end{align}
It is easy to see that $\hrho \| \z^* \|_1 \le | \vbf_{i_d} |$, otherwise by \eqref{proxform1} $ \z^* = 0$, a contradiction. If $\hrho \| \z^* \|_1 = | \vbf_{i_d} |$, then by \eqref{proxform1}, $\| \z^* \|_1 = | \vbf_{i_d} | = 0 \implies \vbf = 0$, a contradiction to our assumption. Therefore, $\hrho \| \z^* \|_1 < | \vbf_{i_d} |$. Suppose that $ | \vbf_{i_k} | \le \hrho \| \z^* \|_1 < | \vbf_{i_{k+1}} |$ for some $k$ s.t. $0 \le k \le d-1$. Note that such $k$ is unique for a fixed $\vbf$. Then by \eqref{proxform1} we have
\begin{align*}
    \z_{i_t}^* = \begin{cases}
        0 & \mbox{if } 1 \le t \le k \\
        \vbf_{i_t} - \sgn(\vbf_{i_t}) \hrho \| \z^* \|_1 & \mbox{if } t > k
    \end{cases}
\end{align*}
This indicates that
\begin{align*}
    & \| \z^* \|_1 = \sum_{t = k+1}^d | \vbf_{i_t} | - \hrho (d - k) \| \z^* \|_1 \implies \| \z^* \|_1 = \frac{1}{\hrho(d-k) +1} \sum_{t = k+1}^d | \vbf_{i_t} |.
\end{align*}
Therefore,
\begin{align}
    \z_{i_t}^* = \begin{cases}
        0 & \mbox{if } 1 \le t \le k \\
        \vbf_{i_t} - \frac{\sgn(\vbf_{i_t}) \hrho}{\hrho(d-k)+1} \sum_{t = k+1}^d | \vbf_{i_t} | & \mbox{if } t > k
    \end{cases}
\end{align}
By the fact that $ | \vbf_{i_k} | \le \hrho \| \z^* \|_1 < | \vbf_{i_{k+1}} | $, this solution leads to $s_k \le \| \vbf \|_1 < s_{k+1}$. By the fact that $s_{t+1} - s_t = \frac{(d-t) + 1}{\hrho} (| \vbf_{i_{t+1}}  | - | \vbf_{i_t} |)  $ and $s_d = \| \vbf \|_1 + \frac{1}{\hrho}| \vbf_{i_d} | $, $s_t$ is non-decreasing and $\| \vbf \|_1 < s_d$. \qed
\end{proof}
An alternative to case 1 is an indicator function defined as follows.\\
Case 2: Let 
\begin{align}\label{case2phi}
    \phi(\z) = \delta_{\bar X} (\z) = \begin{cases}
        0 & \mbox{if } \z \in \bar X \\
        +\infty & \mbox{if } \z \notin \bar X,
    \end{cases}
\end{align}
where $ \bar X \triangleq \{ \z \mid \| \z \|_1 \le \psi \}$ and $\psi > 0$.

In this case, we need to solve the following problem in the proximal projection step.
\begin{align}\label{opt: l1consprox}
\begin{aligned}
\min_{\z} \quad &\dfrac{1}{2} \lVert \z-\vbf \rVert_2^2 \\
\textit{s.t.} \quad &\lVert \z \rVert_1 \leq \psi
\end{aligned}
\end{align}
The next lemma reveals the closed-form solution to \eqref{opt: l1consprox}. Since this result utilizes similar technique as in Lemma~\ref{lm: comp.l12prox}, we relegate its proof to Appendix~\ref{app: proof}.
\begin{lemma}\label{lm: l1consprox}
Suppose that $\z^*$ denotes the optimal solution of \eqref{opt: l1consprox}. If $\| \vbf \|_1 \le \psi$, then $ \z^* = \vbf $. Otherwise, suppose that $i_1,\hdots,i_d$ is a permutation of $\{1,2,...,d \}$ such that $ | \vbf_{i_1}| \ge | \vbf_{i_2}| \ge .... \ge | \vbf_{i_d}|$. Denote $s_0 \triangleq 0$, $s_d \triangleq \| \vbf \|_1$ and $s_m \triangleq \sum_{i=1}^m | \vbf_{i_m} | - m | \vbf_{i_{m+1}} |$, $m = 1,2,\hdots,d-1$. Then $s_m$ is a non-decreasing sequence. Suppose that $s_{\bar m-1} < \psi \le s_{\bar m}$ for some $\bar m \in [d]$. Then we have
\begin{align*}
    & \z_{i_t}^* = \begin{cases}
        0 & \mbox{if }  \bar m < t \le d \\
        \vbf_{(i_t)} - \frac{\sgn\left(\vbf_{(i_t)}\right)}{\bar m}\left(\sum_{k=1}^{\bar m} \left\lvert \vbf_{(i_k)} \right\rvert - \psi \right) & \mbox{if }  1 \le t \le \bar m
    \end{cases}
\end{align*}
\end{lemma}
\begin{remark}
    According to the formulae given by Lemma~\ref{lm: comp.l12prox} and~\ref{lm: l1consprox}, computing $z^*$ of \eqref{opt: l1squareprox} or \eqref{opt: l1consprox} only requires $O(d)$ fundamental operations.
\end{remark}
\subsection{Choices of $G^k$}
Next we discuss choices of the gradient estimate $G^k$ and its variance in infinity norm.\\
{\bf Minibatch.} First we consider using a finite batch of samples to estimate the gradient $g^k \triangleq \nabla f(\x^k)$. Although the batchsize is not necessarily small, we borrow the name \textit{minibatch} considering that accurate estimation of $\nabla f(\x^k)$ will take probably huge or even infinite number of samples. More specifically, let
\begin{align}\label{def: Gk-SGD}
    G^k \triangleq \frac{1}{m} \sum_{i = 1}^m \nabla F(\x^k,\zeta^k_i),
\end{align}
where $\zeta^k_i$ are i.i.d realizations of $\zeta$. Then the following results regarding the sub-Gaussian assumption of the bias hold. Since the results are relatively basic, we relegate the proof to Appendix~\ref{app: proof}.
\begin{lemma}\label{lm: subG} Assumption~\ref{ass: subG} holds and consider the minibatch sampling \eqref{def: Gk-SGD}. Let $\bar w^k \triangleq \frac{1}{m} \sum_{i=1}^{m} (\nabla F(\x^k, \zeta^k_i) - \nabla f(\x^k))$, then for any $k$,
    \begin{itemize}
        \item[1.] $ \bE[\bar w^k \mid \x^k] = 0 $,
        \item[2.] $ \bE[\exp( t \bar w^k_i ) \mid \x^k] \le \exp(t^2 \cdot \sigma_\infty ^2/(2m)) $, for any $t \in \R$, $i = 1,\hdots,d$,
        \item[3.] $ \bE[ \| \bar w^k \|_\infty^2 \mid \x^k] \le 6 \log (2d) \sigma_\infty^2/m $. 
    \end{itemize}
\end{lemma}
\begin{remark}
By \eqref{def: Gk-SGD} and Lemma~\ref{lm: subG}, we have that $\exists c > 0$, s.t., $\forall k$,
\begin{align} \label{redvar-SGD}
\bE[\| g^k - G^k \|_\infty^2] \le \frac{c(\log d)\sigmainf^2}{m}.
\end{align}
\end{remark}

\noindent{\bf Variance Reduction.} Second, we can also apply the variance reduction technique to construct the sampled gradient \cite{johnson2013accelerating}. The technique requires using relatively large number of samples intermittently when estimating $g^k$, and reuse this relatively more accurate estimation in other iterations to improve sampling efficiency. 

Given interval length $q$, define $n_k \triangleq \lfloor \frac{k-1}{q} \rfloor q + 1$. Then $G^k$ is computed as follows:
\begin{align}\label{def: Gk-SVRG}
& G^k \triangleq \\
\notag
& \begin{cases}
    \frac{1}{m_k} \sum_{i=1}^{m_k} \nabla f(\x^k, \zeta^k_i),
    & \mbox{if} \mod(k,q) = 1,\\
    \frac{1}{m_k} \left( \sum_{i=1}^{m_k} \nabla F(\xbf^k,\zeta^k_i) - \sum_{i=1}^{m_k} \nabla F(\xbf^{n_k},\zeta^k_i) \right) + G^{n_k},  & \mbox{otherwise,}
\end{cases}
\end{align}
where $\zeta^k_i$ are i.i.d realizations of $\zeta$. Similar to \eqref{redvar-SGD}, for any $n_k$,
\begin{align}\label{redvar-SVRG1}
\bE [ \| g^{n_k} - G^{n_k} \|_\infty^2 ] \le \frac{c(\log d)\sigma_\infty^2}{m_{n_k}}.    
\end{align}
Suppose that 
\begin{align}\label{mk}
\begin{cases}
    m_k \equiv m & \mbox{if } k \neq n_k \\
m_{n_k} \equiv m_1  & \mbox{otherwise}.
\end{cases}
\end{align} 
Then we have the following result regarding the upper bound for the averaged variance across all the iterates in infinity norm.
\begin{lemma}\label{lm: svrgbd} 
{Consider Algorithm~\ref{alg: DI-SGD}.  Suppose that Assumption~\ref{assumption first-order oracle} to~\ref{ass: Lip} hold.} Consider using variance reduction to estimate $G^k$ (\eqref{def: Gk-SVRG}\eqref{mk}). Then
\begin{align}
\label{redvar-SVRG2}
 \bE[ \| \nabla f(\x^Y) - G^Y \|_\infty^2 ]
    \le \frac{8L^2q^2}{m} \bE[ \| \x^{Y+1} - \x^Y \|^2 ] + \frac{2c(\log d)\sigma_\infty^2}{m_1}.
\end{align}
\end{lemma}
\begin{remark}
We defer the proof to Appendix~\ref{app: proof} since the proof in $\| \cdot \|_\infty$ is similar to the one in $\| \cdot \|$, which is much discussed in literature. Lemma~\ref{lm: svrgbd} indicates that we can bound the average of gradient variance in $\| \cdot \|_\infty$ by summation of two terms, the average square of difference of consecutive iterates and a constant depending only on $\log d$. Note that Lemma~\ref{lm: svrgbd} holds regardless of how we define the proximal projection step (Step 2.2 in Algorithm~\ref{alg: DI-SGD}) as long as the iterates are within $X$. To control the variance to a low level, we only need to pick a large $m_1$. Requirement on the size of $m$ is less restrictive since the average $\bE[\| x^{Y+1} - x^Y \|^2]$ can be controlled at a low level when iteration count $K$ is large. Therefore the sampling efficiency is improved.
\end{remark}

\section{Convergence and complexity analysis}
\label{sec: analysis}
In this section we analyze Algorithm~\ref{alg: DI-SGD} in the various settings discussed in the last section and derive dimension-insensitive sample complexity guarantees. 

First we provide two key lemmas useful for analyzing all the cases.
\subsection{Key lemmas}
Throughout this section, suppose that Assumption~\ref{assumption first-order oracle}-\ref{ass: setw3hj} hold and let $\Delta \triangleq \bE[f(x^1)] - \underline{f}$. We start with the following lemma which holds without specifying the choices of $\phi$ or $G^k$.
\begin{lemma}\label{lm: errbd1} Consider Algorithm~\ref{alg: DI-SGD} with \eqref{ineq: proj}. Then for any positive scalars $\gamma, \tau, t$, we have
\begin{align}
\notag
& \left( \frac{1}{\eta} - \frac{L}{2} - \frac{\gamma}{2 \eta} \right) \bE [ \| \Delta \xbf^{Y+1} \|^2 ] + \frac{1}{\eta} \bE[ ( \xi^{Y+1} )^T  (\ybf^{Y+1} - \xbf^Y) ]
\\
\notag
& \le \frac{\Delta}{K} + \frac{\tau + t}{2} \bE[ \| g^Y - G^Y \|_\infty^2 ] + \left( \frac{1}{2\tau} + \frac{1}{2\eta\gamma} \right) \eh^2  + \frac{1}{2t} \bE [ \| \y^{Y+1} - \x^Y \|_1^2 ] \\
\label{errbd1}
& + \frac{\eh}{ \eta } \bE[ \| \xi^{Y+1} \|_\infty ].
\end{align}
\end{lemma}
\begin{proof}
    We have that 
\begin{align*}
    & f(\xbf^{k+1}) \\
    & \overset{\eqref{ineq: Taylor}}{\le} f(\xbf^k) + (g^k)^T \Delta \xbf^{k+1} + \frac{L}{2} \| \Delta \xbf^{k+1} \|^2 \\
    & \le f(\xbf^k) + (g^k - G^k)^T \Delta \xbf^{k+1} + (G^k)^T \Delta \xbf^{k+1} +  \frac{L}{2} \| \Delta \xbf^{k+1} \|^2 \\
    & \overset{\eqref{ineq: proj1}}{\le} f(\xbf^k) + (g^k - G^k)^T ( \xbf^{k+1} - \ybf^{k+1} ) + (g^k - G^k)^T ( \ybf^{k+1} - \xbf^k ) \\
    & + \frac{1}{\eta}(\Delta \xbf^{k+1} + \xi^{k+1} + \Gamma^{k+1} )^T(\xbf^k - \xbf^{k+1}) +  \frac{L}{2} \| \Delta \xbf^{k+1} \|^2 \\
    & \le f(\xbf^k)  + \frac{\tau}{2} \| g^k - G^k \|_\infty^2 + \frac{1}{2\tau} \| \xbf^{k+1} - \ybf^{k+1} \|_1^2 + \frac{t}{2} \| g^k - G^k \|_\infty^2 \\
    & + \frac{1}{2t} \| \ybf^{k+1} - \xbf^k \|_1^2 - \left( \frac{1}{\eta} - \frac{L}{2} \right) \| \Delta \xbf^{k+1} \|^2 - \frac{1}{\eta} ( \Gamma^{k+1} )^T  \Delta \xbf^{k+1} \\
    & + \frac{1}{\eta} ( \xi^{k+1} )^T  (\xbf^k - \ybf^{k+1}) + \frac{1}{\eta} ( \xi^{k+1} )^T  (\ybf^{k+1} - \xbf^{k+1}) \\
    & \le f(\xbf^k)  + \frac{\tau}{2} \| g^k - G^k \|_\infty^2 + \frac{1}{2\tau} \| \xbf^{k+1} - \ybf^{k+1} \|_1^2 + \frac{t}{2} \| g^k - G^k \|_\infty^2 \\
    & + \frac{1}{2t} \| \ybf^{k+1} - \xbf^k \|_1^2 - \left( \frac{1}{\eta} - \frac{L}{2} \right) \| \Delta \xbf^{k+1} \|^2 + {\frac{\gamma}{2\eta} \| \Delta \xbf^{k+1} \|^2 + \frac{1}{2 \eta \gamma} \| \Gamma^{k+1} \|^2} \\
    & + \frac{1}{\eta} ( \xi^{k+1} )^T  (\xbf^k - \ybf^{k+1}) + \frac{1}{\eta} ( \xi^{k+1} )^T  (\ybf^{k+1} - \xbf^{k+1}) \\
    & \overset{\eqref{ineq: proj3}}{\le} f(\xbf^k)  + \frac{\tau}{2} \| g^k - G^k \|_\infty^2 + \frac{\eh^2}{2\tau} + \frac{t}{2} \| g^k - G^k \|_\infty^2 + \frac{1}{2t} \| \ybf^{k+1} - \xbf^k \|_1^2 \\
    & - \left( \frac{1}{\eta} - \frac{L}{2} - \frac{\gamma}{2\eta} \right) \| \Delta \xbf^{k+1} \|^2 + \frac{\eh^2}{2\eta \gamma} + \frac{1}{\eta} ( \xi^{k+1} )^T  (\xbf^k - \ybf^{k+1}) \\
    & + \frac{1}{\eta}\| \xi^{k+1} \|_\infty \| \y^{k+1} - \x^{k+1} \|_1 \\
    & \le f(\xbf^k)  + \frac{\tau}{2} \| g^k - G^k \|_\infty^2 + \frac{\eh^2}{2\tau} + \frac{t}{2} \| g^k - G^k \|_\infty^2 + \frac{1}{2t} \| \ybf^{k+1} - \xbf^k \|_1^2 \\
    & - \left( \frac{1}{\eta} - \frac{L}{2} - \frac{\gamma}{2\eta} \right) \| \Delta \xbf^{k+1} \|^2 + \frac{\eh^2}{2\eta \gamma} + \frac{1}{\eta} ( \xi^{k+1} )^T  (\xbf^k - \ybf^{k+1}) + \frac{\eh}{\eta}\| \xi^{k+1} \|_\infty \\ 
\end{align*}
By rearranging terms, we have that
\begin{align*}
    & \left( \frac{1}{\eta} - \frac{L}{2} - \frac{\gamma}{2 \eta} \right) \| \Delta \xbf^{k+1} \|^2 + \frac{1}{\eta} ( \xi^{k+1} )^T  (\ybf^{k+1} - \xbf^k) \\
    & \le f(\xbf^k) - f(\xbf^{k+1}) + \frac{\tau + t}{2} \| g^k - G^k \|_\infty^2 + \left( \frac{1}{2\tau} + \frac{1}{2\eta\gamma} \right) \eh^2 + \frac{1}{2t} \| \y^{k+1} - \x^k \|_1^2 \\
    & + \frac{\eh}{\eta} \| \xi^{k+1} \|_\infty.
\end{align*}
Sum up the above inequality for $k = 1,\hdots,K$ and take expectation on both sides, we have
\begin{align*}
   &  \left( \frac{1}{\eta} - \frac{L}{2} - \frac{\gamma}{2 \eta} \right) \sum_{k=1}^K \bE [ \| \Delta \xbf^{k+1} \|^2 ] + \frac{1}{\eta} \sum_{k=1}^K \bE[ ( \xi^{k+1} )^T  (\ybf^{k+1} - \xbf^k) ] \\
    & \le \bE[ f(\xbf^1) - f(\xbf^{K+1}) ] + \frac{\tau + t}{2} \sum_{k=1}^K \bE[ \| g^k - G^k \|_\infty^2 ]
    + \left( \frac{1}{2\tau} + \frac{1}{2\eta\gamma} \right) K \eh^2 \\
    & + \frac{1}{2t} \sum_{k=1}^K \bE [ \| \y^{k+1} - \x^k \|_1^2 ] + \frac{\eh}{\eta} \sum_{k=1}^K \bE[ \| \xi^{k+1} \|_\infty ].
\end{align*}
Recall $\Delta \triangleq \bE[f(\x^1)] - \lf \ge \bE[ f(\xbf^1) - f(\xbf^{K+1}) ]$ and divide both sides by $K$. Then the result follows.\qed
\end{proof}

The residual of interest is ${\rm dist}_{\| \cdot \|_\infty}(0, \partial(f+\delta_X)(\x^{Y+1}))] $. The following lemma provides an upper bound for its expected value.
\begin{lemma} Consider Algorithm~\ref{alg: DI-SGD} with \eqref{ineq: proj}. We have
    \begin{align}
\notag
    & \bE[ {\rm dist}_{\| \cdot \|_\infty}(0, \partial(f+\delta_X)(\x^{Y+1}))] \\
\label{errbd2}
    & \le \left(L + \frac{1}{\eta} \right) \sqrt{\bE[ \| \Delta \x^{Y+1} \|^2 ]} +  \sqrt{\bE[ \| g^Y - G^Y \|_\infty^2 ]} + \frac{1}{\eta} \bE[ \| \xi^{Y+1} \|_\infty ] + \frac{\eh}{\eta}.
\end{align}
\end{lemma}
\begin{proof}
Note that due to \eqref{ineq: proj1}, we have that
\begin{align*}
    \nabla f(\x^{Y+1}) - \frac{1}{\eta}(\Delta \x^{Y+1} + \eta G^Y + \xi^{Y+1} + \Gamma^{Y+1} ) \in \partial(f+\delta_X)(\x^{Y+1}). 
\end{align*}
Therefore, 
\begin{align}
\notag
    & \bE[ {\rm dist}_{\| \cdot \|_\infty}(0, \partial(f+\delta_X)(\x^{Y+1}))] \\
    \notag
    & \le \bE[\| \nabla f(\x^{Y+1}) - \frac{1}{\eta}(\Delta \x^{Y+1} + \eta G^Y + \xi^{Y+1} + \Gamma^{Y+1} ) \|_\infty] \\
    \notag
    & \le \bE[ \| \nabla f(\x^{Y+1}) - g^Y \|_\infty ]  + \bE[ \| g^Y - G^Y \|_\infty] \\
    \notag
    & + \frac{1}{\eta} \bE[ \| \Delta \x^{Y+1} + \xi^{Y+1} + \Gamma^{Y+1} \|_\infty ] \\\notag
    & \le L \bE[ \| \Delta \x^{Y+1} \| ] + \bE[ \| g^Y - G^Y \|_\infty ] \\
    \notag
    & + \frac{1}{\eta} (\bE[ \| \Delta \x^{Y+1} \| ] + \bE[ \| \xi^{Y+1} \|_\infty] + \bE[ \| \Gamma^{Y+1} \|_\infty ] ) \\\notag
    & \le \left(L + \frac{1}{\eta} \right) \bE[ \| \Delta \x^{Y+1} \| ] + \bE[ \| g^Y - G^Y \|_\infty ] + \frac{1}{\eta} \bE[ \| \xi^{Y+1} \|_\infty ] + \frac{\eh}{\eta}\\
    \notag
    & \le \left(L + \frac{1}{\eta} \right) \sqrt{\bE[ \| \Delta \x^{Y+1} \|^2 ]} +  \sqrt{\bE[ \| g^Y - G^Y \|_\infty^2 ]} + \frac{1}{\eta} \bE[ \| \xi^{Y+1} \|_\infty ] + \frac{\eh}{\eta}.
\end{align}\qed
\end{proof}
\subsection{Analysis for Case 1}
The main result for minibatch sampling is as follows.
\begin{theorem}\label{thm: case1mini} For Algorithm~\ref{alg: DI-SGD} with \eqref{ineq: proj}, consider case 1 \eqref{setting1} and estimating $G^k$ using minibatch \eqref{def: Gk-SGD}. Then
\begin{align}
\notag
    & \bE[ {\rm dist}_{\| \cdot \|_\infty}(0, \partial(f+\delta_X)(\x^{Y+1})) ] \\
\notag
    & \le  \left( L + \frac{1}{\eta} \right) \sqrt{\frac{\frac{\Delta}{K} + \frac{(\tau + t)c(\log d)\sigmainf^2}{2m} + \left( \frac{1}{2\tau} + \frac{1}{2\eta\gamma} + \frac{2\hrho}{\eta} \right) \eh^2}{1/\eta - L/2 - \gamma/(2\eta)}} +  \sqrt{ \frac{c(\log d)\sigmainf^2}{m} } \\
    \label{errbd4}
    & + \frac{\hrho}{\eta}\sqrt{\frac{\frac{\Delta}{K} + \frac{(\tau + t)c(\log d)\sigmainf^2}{2m} + \left( \frac{1}{2\tau} + \frac{1}{2\eta\gamma} + \frac{2\hrho}{\eta} \right) \eh^2}{\hrho/(8\eta)}} + \frac{ \eh}{\eta},
\end{align}    
holds when $1/\eta - L/2 - \gamma/(2\eta) > 0$, $t = 2\eta/\hrho$, $\tau > 0$, $\gamma > 0$.
\end{theorem}
\begin{proof}
By Lemma~\ref{lm: l1subdiff}
we have that
\begin{align*}
    \partial \phi(\z - \x^k) \subseteq \{ \hrho \| \z - \x^k \|_1 v \mid v_i \in [-1,1] \}.
\end{align*}
Therefore, 
\begin{align}\label{subgbd}
    \| \xi^{k+1} \|_\infty \le \hrho \| \y^{k+1} - \x^k \|_1
\end{align}
Note that $\phi(0) = 0$, then by \eqref{eq: proj2} and convexity of $\phi$,
\begin{align}\label{ineq: phi}
    ( \xi^{k+1} )^T  (\xbf^k - \ybf^{k+1}) + \phi(\ybf^{k+1}-\x^k) \le \phi(\x^k - \x^k) = \phi(0) = 0.
\end{align}

Since we consider using a minibatch of samples to estimate the gradient \eqref{def: Gk-SGD}. Then \eqref{errbd1}\eqref{setting1}\eqref{subgbd}\eqref{ineq: phi} lead to
\begin{align}
\notag
& \left( \frac{1}{\eta} - \frac{L}{2} - \frac{\gamma}{2 \eta} \right) \bE [ \| \Delta \xbf^{Y+1} \|^2 ] + \frac{\hrho}{2\eta} \bE[ \| \ybf^{Y+1} - \xbf^Y \|_1^2 ] \\\notag
& \le \left( \frac{1}{\eta} - \frac{L}{2} - \frac{\gamma}{2 \eta} \right) \bE [ \| \Delta \xbf^{Y+1} \|^2 ] + \frac{1}{\eta} \bE[ (\xi^{Y+1})^T (  \ybf^{Y+1} - \xbf^Y ) ] \\\notag
& \le \frac{\Delta}{K} + \frac{\tau + t}{2} \bE[ \| g^Y - G^Y \|_\infty^2 ] + \left( \frac{1}{2\tau} + \frac{1}{2\eta\gamma} \right) \eh^2 + \frac{1}{2t} \bE [ \| \y^{Y+1} - \x^Y \|_1^2 ] \\
\notag
& + \frac{\eh}{ \eta } \bE[ \| \xi^{Y+1} \|_\infty ] \\ \notag
& \le \frac{\Delta}{K} + \frac{\tau + t}{2} \bE[ \| g^Y - G^Y \|_\infty^2 ] + \left( \frac{1}{2\tau} + \frac{1}{2\eta\gamma} + \frac{2\hrho}{\eta} \right) \eh^2 + \frac{1}{2t} \bE [ \| \y^{Y+1} - \x^Y \|_1^2 ] \\
\notag
& + \frac{1}{ 8 \eta \hrho } \bE[ \| \xi^{Y+1} \|_\infty^2 ] \\ \notag
& \overset{\left(t = \frac{2\eta}{\hrho}\right)}{=} \frac{\Delta}{K} +  \frac{\tau + t}{2}  \bE[ \| g^Y - G^Y \|_\infty^2 ] + \left( \frac{1}{2\tau} + \frac{1}{2\eta\gamma} + \frac{2\hrho}{\eta} \right) \eh^2 \\
\notag
& + \frac{\hrho}{4\eta}  \bE[ \| \ybf^{Y+1} - \xbf^Y \|_1^2 ] +  \frac{\hrho}{8\eta} \bE[ \| \ybf^{Y+1} - \xbf^Y \|_1^2 ] \\
\notag
& \implies \left( \frac{1}{\eta} - \frac{L}{2} - \frac{\gamma}{2 \eta} \right) \bE [ \| \Delta \xbf^{Y+1} \|^2 ] + \frac{\hrho}{8\eta} \bE[ \| \ybf^{Y+1} - \xbf^Y \|_1^2 ] \\\notag
& \le \frac{\Delta}{K} +  \frac{\tau + t}{2}  \bE[ \| g^Y - G^Y \|_\infty^2 ] + \left( \frac{1}{2\tau} + \frac{1}{2\eta\gamma} + \frac{2\hrho}{\eta} \right) \eh^2 \\
& \overset{\eqref{redvar-SGD}}{\le} \frac{\Delta}{K} + \frac{(\tau + t)c(\log d)\sigmainf^2}{2m} + \left( \frac{1}{2\tau} + \frac{1}{2\eta\gamma} + \frac{2\hrho}{\eta} \right) \eh^2.
\label{errbd3}
\end{align}
By \eqref{redvar-SGD}\eqref{errbd2}\eqref{subgbd}\eqref{errbd3},
the result follows.\qed
\end{proof}
\begin{corollary}
Under the setting of Theorem~\ref{thm: case1mini}, let
\begin{align*}
    \eta = \frac{1}{L},\; K = \left\lceil \frac{\Delta L}{\epsilon^2} \right\rceil,\; m = \left\lceil \frac{c(\log d) \sigma_\infty^2}{\epsilon^2} \right\rceil,\; \gamma = \frac{1}{2},\; \tau = \frac{1}{L},\; \eh = \frac{\epsilon}{L}, \; t = \frac{2\eta}{\hrho}.
\end{align*}
Then by \eqref{errbd4} we have
\begin{align*}
\bE[ {\rm dist}_{\| \cdot \|_\infty}(0, \partial(f+\delta_X)(\x^{Y+1})) ] \le ( 4\sqrt{3+1/\hrho + 2 \hrho} + 2 \sqrt{4\hrho^2 + 6\hrho + 2} + 2 )\epsilon.
\end{align*}
The sample complexity of the algorithm is
\begin{align*}
    Km = \left\lceil \frac{\Delta L}{\epsilon^2} \right\rceil \left\lceil \frac{c(\log d) \sigma_\infty^2}{\epsilon^2} \right\rceil = \sO \left(\frac{\Delta L (\log d) \sigma_\infty^2}{\epsilon^4} \right).
\end{align*}
\end{corollary}
The next theorem reveals the convergence rate by applying variance reduction in Case 1. The proof is similar to Theorem~\ref{thm: case1mini} and relegated to Appendix~\ref{app: proof}.
\begin{theorem}\label{thm: case1svrg}
For Algorithm~\ref{alg: DI-SGD} with \eqref{ineq: proj}, consider case 1 \eqref{setting1} and estimating $G^k$ using variance reduction \eqref{def: Gk-SGD}\eqref{mk}. Then the following holds for $1/\eta - L/2 - \gamma/(2\eta) - \frac{4(\tau + t)L^2q^2}{m} > 0$, $t = 2\eta/\hrho$, $\tau > 0$, $\gamma > 0$.
\begin{align}
\notag
    & \bE[ {\rm dist}_{\| \cdot \|_\infty}(0, \partial(f+\delta_X)(\x^{Y+1})) ] \\
    \notag
    & \le \left( L + \frac{1}{\eta} + \sqrt{\frac{8L^2q^2}{m}} \right) \sqrt{\frac{\frac{\Delta}{K} + \frac{(\tau + t)c(\log d)\sigmainf^2}{m_1} + \left( \frac{1}{2\tau} + \frac{1}{2\eta\gamma} + \frac{2\hrho}{\eta} \right) \eh^2}{\frac{1}{\eta} - \frac{L}{2} - \frac{\gamma}{2\eta} - \frac{4(\tau + t)L^2q^2}{m} }}  \\ 
\label{errbd8}
    & + \sqrt{ \frac{2c(\log d)\sigma_\infty^2}{m_1} } + \frac{\hrho}{\eta}\sqrt{\frac{\frac{\Delta}{K} + \frac{(\tau + t)c(\log d)\sigmainf^2}{m_1} + \left( \frac{1}{2\tau} + \frac{1}{2\eta\gamma} + \frac{2\hrho}{\eta} \right) \eh^2}{\hrho/(8\eta)}} + \frac{\eh}{\eta}.
\end{align}
\end{theorem}
\begin{corollary}
Under the setting of Theorem~\ref{thm: case1svrg}, let
\begin{align*}
    & \eta = \frac{1}{L},\; K = \left\lceil \frac{\Delta L}{\epsilon^2} \right\rceil,\; m_1 = \left\lceil \frac{c(\log d) \sigma_\infty^2}{\epsilon^2} \right\rceil,\; m = q^2 = \left\lceil \left( \frac{c(\log d)\sigma^2_\infty}{\epsilon^2} \right)^{1/3} \right\rceil^2,\; \\
    & \gamma = \frac{1}{2},\; \tau = t = \frac{1}{64L},\;
    \eh = \frac{\epsilon}{L},\; \hrho = 128.
\end{align*}
Then by \eqref{errbd8} we have
\begin{align*}
\bE[ {\rm dist}_{\| \cdot \|_\infty}(0, \partial(f+\delta_X)(\x^{Y+1})) ] {\le} C\epsilon,
\end{align*}
where $C$ is an absolute constant independent of any parameter of the algorithm or the problem data. The sample complexity is upper bounded by
\begin{align*}
   Km + \left\lceil \frac{K}{q} \right\rceil m_1 & \le  2K (m + q^2)  = 4 \left\lceil \frac{\Delta L}{\epsilon^2} \right\rceil \left\lceil \left( \frac{c(\log d)\sigma^2_\infty}{\epsilon^2} \right)^{1/3} \right\rceil^2 \\
   & = \sO\left( \frac{ \Delta L(c(\log d)\sigma_\infty^2)^{2/3}}{\epsilon^{10/3}} \right).
\end{align*}
\end{corollary}
\subsection{Analysis for Case 2}
The main result is given as follows.
\begin{theorem}\label{thm: case2mini} For Algorithm~\ref{alg: DI-SGD} with \eqref{ineq: proj}, consider case 2 \eqref{setting2} and estimating $G^k$ using minibatch \eqref{def: Gk-SGD}. Then
\begin{align}
\notag
    & \bE[ {\rm dist}_{\| \cdot \|_\infty}(0, \partial(f+\delta_X)(\x^{Y+1})) ] \\
\notag
    & \le \left( L + \frac{1}{\eta} \right) \sqrt{\frac{\frac{\Delta}{K} + \frac{(\tau + t)c(\log d)\sigmainf^2}{2m} + \left( \frac{1}{2\tau} + \frac{1}{2\eta\gamma} \right) \eh^2 + \frac{\psi^2}{2t}}{1/\eta - L/2 - \gamma/(2\eta)}} + \sqrt{ \frac{c(\log d)\sigmainf^2}{m} } \\
\label{errbd6}
    & + \frac{1}{\eta} \cdot \frac{\frac{\Delta}{K} + \frac{(\tau + t)c(\log d)\sigmainf^2}{2m} + \left( \frac{1}{2\tau} + \frac{1}{2\eta\gamma} \right) \eh^2 + \frac{\psi^2}{2t}}{(\psi - \eh)/\eta} + \frac{\eh}{\eta}.
\end{align}
holds when $1/\eta - L/2 - \gamma/(2\eta) > 0$, $t >0$, $\tau > 0$, $\gamma > 0$, $\psi > \eh$.
\end{theorem}
\begin{proof}
Note that \eqref{eq: proj2} is equivalent to
\begin{align}
    \y^{k+1} \in \mbox{arg}\min_\y -(\xi^{k+1})^T \y \quad \mbox{s.t. } \quad \| \y - \x^k \|_1 \le \psi.
\end{align}
Therefore, $\exists \hrho_k$ such that
\begin{align}\label{setting2}
\xi^{k+1} \in \hrho_k \partial (\| \cdot - \x^k \|_1)(\y^{k+1}), \quad 0 \le \hrho_k \perp \| \y^{k+1} - \x^k \|_1 - \psi \le 0.
\end{align}
Then
\begin{align*}
    \xi^{k+1} \in \{ \hrho_k v \mid v_i \in [-1,1] \}.
\end{align*}
Therefore, 
\begin{align}\label{subgbd2}
    \| \xi^{k+1} \|_\infty \le \hrho_k.
\end{align}
Moreover, by \eqref{setting2},
\begin{align}\label{ineq: nn3}
    (\xi^{k+1})^T(\y^{k+1} - \x^k) \ge \hrho_k \| \y^{k+1} - \x^k \|_1 - \hrho_k \| \x^k - \x^k \|_1 = \hrho_k \| \y^{k+1} - \x^k \|_1.
\end{align}

Consider the minibatch sampling \eqref{def: Gk-SGD}. Then \eqref{errbd1}\eqref{setting2}\eqref{subgbd2}\eqref{ineq: nn3} lead to
\begin{align}
\notag
& \left( \frac{1}{\eta} - \frac{L}{2} - \frac{\gamma}{2 \eta} \right) \bE [ \| \Delta \xbf^{Y+1} \|^2 ] + \frac{1}{\eta} \bE[ \hrho_Y \| \ybf^{Y+1} - \xbf^Y \|_1 ] \\\notag
& \le \frac{\Delta}{K} + \frac{\tau + t}{2} \bE[ \| g^Y - G^Y \|_\infty^2 ] + \left( \frac{1}{2\tau} + \frac{1}{2\eta\gamma} \right) \eh^2 + \frac{1}{2t} \bE [ \| \y^{Y+1} - \x^Y \|_1^2 ] \\
\notag
& + \frac{\eh}{ \eta } \bE[ \| \xi^{Y+1} \|_\infty ] \\ \notag
& \le \frac{\Delta}{K} + \frac{\tau + t}{2} \bE[ \| g^Y - G^Y \|_\infty^2 ] + \left( \frac{1}{2\tau} + \frac{1}{2\eta\gamma} \right) \eh^2 + \frac{\psi^2}{2t} + \frac{\eh }{ \eta } \bE[\hrho_Y] \\ \notag
& \implies \left( \frac{1}{\eta} - \frac{L}{2} - \frac{\gamma}{2 \eta} \right) \bE [ \| \Delta \xbf^{Y+1} \|^2 ] + \frac{\psi - \eh}{\eta} \bE[ \hrho_Y ] \\\notag
& \le \frac{\Delta}{K} +  \frac{\tau + t}{2}  \bE[ \| g^Y - G^Y \|_\infty^2 ] + \left( \frac{1}{2\tau} + \frac{1}{2\eta\gamma} \right) \eh^2 + \frac{\psi^2}{2t} \\
& \overset{\eqref{redvar-SGD}}{\le} \frac{\Delta}{K} + \frac{(\tau + t)c(\log d)\sigmainf^2}{2m} + \left( \frac{1}{2\tau} + \frac{1}{2\eta\gamma} \right) \eh^2 + \frac{\psi^2}{2t}.
\label{errbd5}
\end{align}
By \eqref{redvar-SGD}\eqref{errbd2}\eqref{subgbd2}\eqref{errbd5}, the result follows. \qed
\end{proof}
\begin{corollary}
    Under the settings of Theorem~\ref{thm: case2mini}, let
\begin{align*}
    & \eta = \frac{1}{L},\; K = \left\lceil \frac{\Delta L}{\epsilon^2} \right\rceil,\; m = \left\lceil \frac{c(\log d) \sigma_\infty^2}{\epsilon^2} \right\rceil,\; \gamma = \frac{1}{2},\; \tau = \frac{1}{L},\; \eh = \frac{\epsilon}{L},\; \psi = 2 \eh,\; \\
    & t = \frac{1}{L}.
\end{align*}
Then by \eqref{errbd6} we have
\begin{align*}
\bE[ {\rm dist}_{\| \cdot \|_\infty}(0, \partial(f+\delta_X)(\x^{Y+1})) ] {\le} C \epsilon,
\end{align*}
where $C$ is an absolute constant. The sample complexity is
\begin{align*}
    Km = \sO\left( \frac{ \Delta L c (\log d) \sigma_\infty^2 }{\epsilon^4} \right).
\end{align*}
\end{corollary}
Apply variance reduction to case 2, and the following statements hold.
\begin{theorem}\label{thm: case2svrg}
For Algorithm~\ref{alg: DI-SGD} with \eqref{ineq: proj}, consider case 2 \eqref{setting2} and estimating $G^k$ using variance reduction \eqref{def: Gk-SGD}\eqref{mk}. Then the following holds when $1/\eta - L/2 - \gamma/(2\eta) - \frac{4(\tau + t)L^2q^2}{m} > 0$, $t > 0$, $\tau > 0$, $\gamma > 0$, $\psi> \eh$.
\begin{align}
\notag
    & \bE[ {\rm dist}_{\| \cdot \|_\infty}(0, \partial(f+\delta_X)(\x^{Y+1})) ] \\ 
    \notag
    & \le \left( L + \frac{1}{\eta} + \sqrt{\frac{8L^2q^2}{m}} \right) \sqrt{\frac{\frac{\Delta}{K} + \frac{(\tau + t)c(\log d)\sigmainf^2}{m_1} + \left( \frac{1}{2\tau} + \frac{1}{2\eta\gamma} \right) \eh^2 + \frac{\psi^2}{2t} }{\frac{1}{\eta} - \frac{L}{2} - \frac{\gamma}{2\eta} - \frac{4(\tau + t)L^2q^2}{m} }} \\
    \label{errbd10}
    & + \frac{1}{\psi - \eh} \left( \frac{\Delta}{K} + \frac{(\tau + t)c(\log d)\sigmainf^2}{m_1} + \left( \frac{1}{2\tau} + \frac{1}{2\eta\gamma} \right) \eh^2 + \frac{\psi^2}{2t} \right) + \frac{\eh}{\eta} \\
    \notag
    & + \sqrt{ \frac{2c(\log d)\sigma_\infty^2}{m_1} }.
\end{align}
\end{theorem}
\begin{proof}
    See Appendix~\ref{app: proof}.
\end{proof}
\begin{corollary}
Under the settings of Theorem~\ref{thm: case2svrg}, let
\begin{align*}
    & \eta = \frac{1}{L},\; K = \left\lceil \frac{\Delta L}{\epsilon^2} \right\rceil,\; m_1 = \left\lceil \frac{c(\log d) \sigma_\infty^2}{\epsilon^2} \right\rceil,\; m = q^2 = m_1^{2/3},\; \gamma = \frac{1}{2},\; \\
    & \tau = t = \frac{1}{64L},\; \eh = \frac{\epsilon}{L},\; \psi = 2 \eh,
\end{align*}
Then by \eqref{errbd10} we have
\begin{align*}
\bE[ {\rm dist}_{\| \cdot \|_\infty}(0, \partial(f+\delta_X)(\x^{Y+1})) ] {\le} C\epsilon,
\end{align*}
where $C$ is an absolute constant independent of any parameter of the algorithm or the problem data. The sample complexity is upper bounded by
\begin{align*}
   \sO\left( \frac{ \Delta L(c(\log d)\sigma_\infty^2)^{2/3}}{\epsilon^{10/3}} \right).
\end{align*}
\end{corollary}

\section{Numerical Experiment}\label{section:num}
In this section, we present some numerical experiments to illustrate the dimension insensitive property of our proposed methods.\footnotemark \footnotetext{Data and codes have been uploaded to   \texttt{https://github.com/gxybrh/DISFOM}} 

\textbf{Problem setting} \quad We consider a nonconvex stochastic quadratic programming problem as follows:
\begin{align}
    \min_{\x \in X} \quad f(\x) \triangleq \frac{1}{2} \mathbb{E}\left[({\bm \alpha}^T \x - b)^2\right] + \lambda\sum_{i=1}^d \frac{\x_i^2}{1 + \x_i^2},
\end{align}
where the feasible region $X = [-R, R]^d$, $(\bm{\alpha}, b)\in \R^d\times\R $ is a random pair that satisfies linear relationship $b = \bm{\alpha}^T \x_{\rm true} + w $. We generate the scaled covariance matrix for $\bm{\alpha}$ by the following procedure. First we generate a $d\times d$ identity matrix, and then replace its top-left $\dfrac{d}{16} \times \dfrac{d}{16}$ principal sub-matrix by $\Sigma_{\rm sub} := \bm{QDQ}^T \in \R^{\frac{d}{16} \times \frac{d}{16}}$, where $\bm{Q}$ consists of the orthonormal basis of a $\dfrac{d}{16} \times \dfrac{d}{16}$ matrix whose entries are i.i.d. uniformly distributed on $(0,1)$ and $\bm{D}$ is a diagonal matrix with each diagonal entry i.i.d. uniformly distributed on $(1,2)$. Let $\bm \alpha = \Sigma^{\frac{1}{2}} s$, where $s_i$, $i = 1,...,d$ are i.i.d and obey truncated standard normal distribution over $[-u,u]$. $w$ also obeys a truncated standard normal distribution over $[-u,u]$ and is independent of $\bm \alpha$.tygj Therefore, the variance $\sigma^2$ of $s_i$ and $w$ has the following formula
\begin{align}
    \sigma^2 = 1 - \frac{\frac{2u}{\sqrt{2\pi}} \exp\left(-\frac{u^2}{2} \right) }{\Phi(u) - \Phi(-u)},
\end{align}
where $\Phi$ is the CDF of standard normal distribution. By construction, function $f$ has the following closed form:
\begin{align*}
    f(\x) = \frac{\sigma^2}{2}(\x - \x_{\rm true})^T \Sigma (\x - \x_{\rm true}) + \lambda\sum_{i=1}^d \dfrac{\x_i^2}{1 + \x_i^2} + \frac{\sigma^2}{2}.
\end{align*}
Moreover, the error of sampled gradient $\nabla_{\x} F(\x,\alpha,w) - \nabla f(\x)$ is sub-Gaussian since it has bounded support. The Lipschitz constant $L$ of $\nabla f$ is $\lambda_{\max} + 2\lambda$, where $\lambda_{\max}$ is the largest eigenvalue of $\sigma^2 \Sigma$, ranging from $\sigma^2$ to $2\sigma^2$ by construction. Therefore, $L$ has a fixed range independent of $d$. $f$ is nonconvex when the minimal eigenvalue of $\nabla^2 f$, $\lambda_{\min} - \lambda/2$ is negative, where $\lambda_{\min}$ denotes the minimal eigenvalue of $\sigma^2 \Sigma$ and equals to $\sigma^2$. 

The experiments present the quality of solutions, measured by the averaged gap $(f - f^*) / \Delta$ and residual out of 3 replications. The value of $f^*$ is given by using the projected gradient method with backtracking (see Algorithm \ref{alg: GD backtracking} in Appendix~\ref{app: alg}), to solve the correlated closed-form problem, and $\Delta \triangleq f(\x^1) - f^*$. We carry on our experiments with dimension $d \in \{2^7, 2^8, \dots, 2^{14}\}$, {$\x_{\rm true} = (1,...,1,0,...,0)^T$ (first $d/16$ elements are 1)} and start with $\xbf_1 = \bm{0}$. $R = 3$, $u = 3$, $\lambda = 2.5$. 

\textbf{Gradient generation} \quad According to the two estimating methods of gradient mentioned above, mini-batch and variance reduction, we compare three algorithms with both methods of gradient estimation. In the mini-batch setting, the batch size $m_k \equiv m = 1000$ and algorithms stop with $K = 300$. In the variance reduction setting, we set $m = 1000, q = 9$. The batch size is $m_k = m = 1000$ when $\mod(k,q) = 1$ and $m_k = m^{\frac{2}{3}} = 100$ otherwise, and the algorithms stop with $K = 1350$.

\textbf{Methods} Here we present the detailed implementation of DISFOMs, proximal stochastic gradient descent (SGD), proximal stochastic variance reduced gradient (SVRG) 
and stochastic mirror descent (SMDs).\begin{enumerate}

\item \textbf{DISFOM$\_$minibatch} \quad Algorithm \ref{alg: DI-SGD} with $\phi(\zbf) = \dfrac{\hat{\rho}}{2} \lVert \zbf \rVert_1^2 $ (Case 1) and minibatch sampling \eqref{def: Gk-SGD}. We choose $\hat{\rho} = 2$ in minibatch.
The stepsize is $\eta_k \equiv \eta = \dfrac{1}{L}$. 

\item \textbf{DISFOM$\_$svrg} \quad Algorithm \ref{alg: DI-SGD} with $\phi(\zbf)  = \dfrac{\hat{\rho}}{2} \lVert \zbf \rVert_1^2 $ (Case 1) and variance reduction \eqref{def: Gk-SGD}. We choose $\hat{\rho} = 128$. 
The stepsize is $\eta_k \equiv \eta = \dfrac{1}{L}$. 

\item (Proximal) \textbf{SGD} \quad Algorithm \ref{alg: DI-SGD} with $P_X^k \equiv P_X$ (Euclidean projection on $X$). $G^k$ is estimated using minibatch \eqref{def: Gk-SGD}. The step size is $\eta_k \equiv \eta = \dfrac{1}{L}$.

\item (Proximal) \textbf{SVRG} \quad Algorithm \ref{alg: DI-SGD} with $P_X^k \equiv P_X$ (Euclidean projection on $X$). $G^k$ is estimated using variance reduction \eqref{def: Gk-SVRG}. The step size is $\eta_k \equiv \eta = \dfrac{1}{10L}$.\footnotemark \footnotetext{Here we need to decrease the stepsize of SVRG to stabilize its performance.}

\item \textbf{SMD$\_$minibatch} \quad Algorithm \ref{alg: DI-SGD}. 
Solve the proximal projection problem in Step 2.2 as follows:
 \begin{align*}
     \xbf^{k+1} = \mbox{arg}\min_{\zbf\in X} \left\{\langle G^k, \zbf\rangle + \dfrac{1}{\alpha_k} D_{\omega}(\zbf, \xbf^k) \right\}
 \end{align*}
 where $D_{\omega}(x,y) = \omega(x) - \omega(y) - \langle\nabla \omega(y), x - y \rangle $
Let the distance generating function $\omega(x) = \frac{C}{2}\lVert x\rVert_p^2$ where $p = 1 + \dfrac{1}{\ln d}$ and $C = e^2\ln d$ (so that $\omega(x)$ is 1-strongly convex w.r.t. $\| \cdot \|_1$), see \cite{beck2003mirror} and the reference therein. The step size is $\alpha_k \equiv \alpha = \dfrac{c}{\sqrt{K}}$, where $c = \sqrt{\dfrac{f(x^1)}{\rho  L^2}}$, $\rho = \dfrac{\lambda}{2} - \lambda_{\min}$, given by \cite{zhang2018convergence}. $G^k$ is estimated via minibatch.

\item \textbf{SMD$\_$svrg} \quad Same settings as in SMD$\_$minibatch, except that $G^k$ is estimated via variance reduction.

\end{enumerate}

\textbf{Results and interpretation}
\begin{figure}[ht]
    \centering
    \includegraphics[width = 0.45\textwidth]{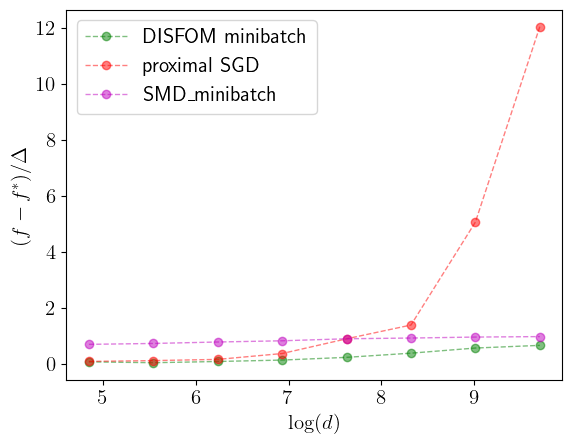}
    \includegraphics[width = 0.45\textwidth]
    {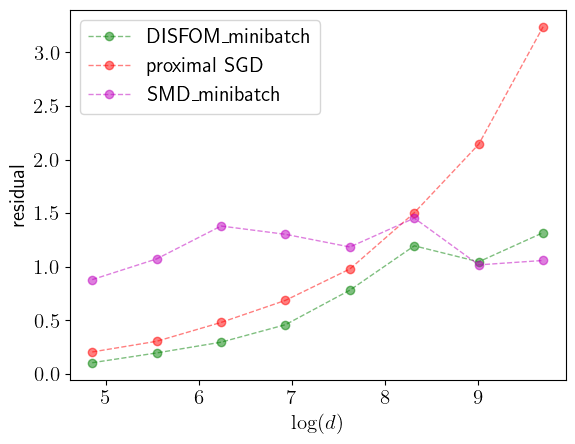}
    \caption{averaged gap $(f-f^*) / \Delta$ vs $\log(d)$ and residual vs $\log(d)$ for mini-batch}\label{fig:minibatch}
\end{figure}
\begin{figure}[ht]
    \centering
    \includegraphics[width = 0.45\textwidth]{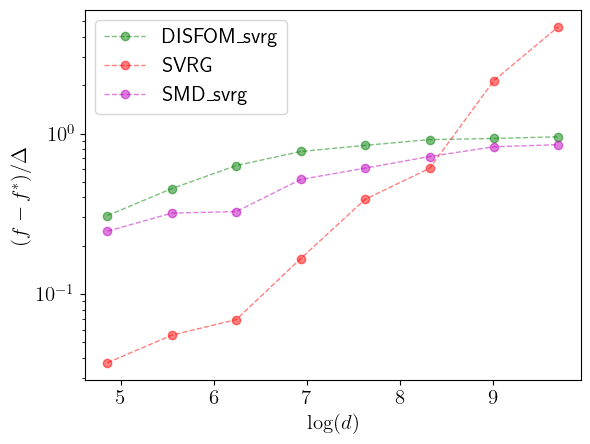}
    \includegraphics[width = 0.45\textwidth]{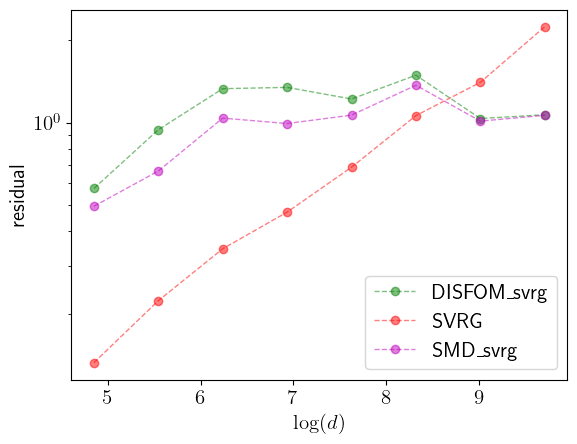}
    \caption{averaged gap $(f-f^*) / \Delta$ vs $\log(d)$ and residual vs $\log(d)$ for variance reduction}\label{fig:svrg}
\end{figure}

From Figure~\ref{fig:minibatch} and \ref{fig:svrg} we can see that the performance of vanilla SGD and SVRG would deteriorate rapidly with the increase of $d$. On the contrary, the performance of DISFOMs and SMDs are comparable and less sensitive to $d$, in terms of both the averaged relative function value gap $(f- f^*)/\Delta$ and the averaged residual $r(\x)$ defined in \eqref{def: resfunc}. 
To our best knowledge, there is not a proper theory for the dimension-insensitive property of SMD$\_$svrg in literature, and the numerical experiments imply this empirically. 
The results for Case 2 are similar to Case 1 thus suppressed.

\section{Conclusion}\label{sec: conclusion}

In this study, we introduce dimension-insensitive stochastic first-order methods (DISFOMs) as a solution for nonconvex optimization with an expected-valued objective function. Our algorithms are designed to accommodate non-Euclidean and non-smooth distance functions as proximal terms. When dealing with a stochastic nonconvex optimization problem, the sample complexity of stochastic first-order methods can become linearly dependent on the problem dimension. This causes trouble for solving large-scale problems. We demonstrate that DISFOM achieves a sample complexity of $O((\log d) / \epsilon^4)$ in order to obtain an $\epsilon$-stationary point via minibatch sampling. Furthermore, we show that DISFOM with variance reduction can enhance this bound to $O((\log d)^{2/3} / \epsilon^{10/3})$. We present two options for the non-smooth distance functions, both of which allow for closed-form solutions in the proximal steps. Preliminary numerical experiments illustrate the dimension-insensitive property of our proposed methods. In our future work, we intend to extend our results to the case when the objective function $f$ is non-smooth. 

		
\bibliographystyle{spmpsci}      
\bibliography{ref} 

	

\appendix

\section{Proofs}\label{app: proof}
Proof of Lemma~\ref{lm: l1consprox}.
\begin{proof}
The corresponding KKT-condition is as follows\begin{align}
\nabla_{\z} L(\z, \lambda) &= (\z - \vbf) + \lambda \partial\lVert \z \rVert_1 = 0 \label{eq:1}\\
g(x) &=  \lVert \z \rVert_1 - \psi \leq 0 \\
\lambda &\geq 0 \\
\lambda g(x) &= 0
\end{align}
From \eqref{eq:1} we have \begin{align*}
\z_i - \vbf_i + \lambda = 0 & \quad \text{if} \quad  \z_i > 0 \\
\z_i - \vbf_i - \lambda = 0 & \quad \text{if} \quad  \z_i < 0 \\
-\lambda \leq \z_i - \vbf_i \leq \lambda & \quad \text{if} \quad  \z_i = 0 \\
\end{align*}
Thus we have \begin{align}
\z_i = \vbf_i - \lambda & \quad \text{if} \quad \vbf_i - \lambda > 0 \label{eq:x1} \\
\z_i = \vbf_i + \lambda & \quad \text{if} \quad \vbf_i + \lambda < 0 \label{eq:x2} \\
\z_i = 0 & \quad \text{if} \quad \lvert \vbf_i \rvert \leq \lambda \label{eq:x3}
\end{align}
Suppose $\lambda = 0$, then $\z = \vbf$. This is the solution if $ \| \vbf \|_1 \le \psi$. Otherwise suppose that $\lambda > 0$. Then we have \begin{align} 
\lVert \z \rVert_1 = \sum_{i}^{\lvert \vbf_i \rvert > \lambda} \lvert \z_i \rvert 
= \sum_{i}^{\vbf_i > \lambda} (\vbf_i - \lambda) + \sum_{i}^{\vbf_i < -\lambda}(-\vbf_i - \lambda) = 
\sum_{i}^{\lvert \vbf_i\rvert > \lambda} (\lvert \vbf_i \rvert - \lambda) = \psi
\end{align}
Suppose there exists $m$ such that $s_{m-1} < \psi \le s_m$, then we can deduce that $| \vbf_{i_m} | > \lambda  \ge | \vbf_{i_{m+1}} |$, therefore,
\begin{align}
    \psi = \sum_{t=1}^m \lvert \vbf_{i_t}\rvert - m\lambda \implies 
    \lambda = \dfrac1{m}\left(\sum_{t=1}^{m} \left\lvert \vbf_{i_t} \right\rvert - \psi \right). \label{eq:v}
\end{align}
Finally we could derive $\z$ from (\ref{eq:x1}), (\ref{eq:x2}), (\ref{eq:x3}), (\ref{eq:v}) as follows.
\begin{align}
    \z_{i_t} = \begin{cases}
        0 & \mbox{if }  m+1 \le t \le d \\
        \vbf_{i_t} - \frac{\sgn\left(\vbf_{i_t}\right)}{m}\left(\sum_{k=1}^{m} \left\lvert \vbf_{i_k} \right\rvert - \psi \right) & \mbox{if }  1 \le t \le m
    \end{cases}
\end{align} \qed
\end{proof}
Proof of Lemma~\ref{lm: subG}.
\begin{proof}
    (1) is obvious from Assumption~\ref{ass: subG}. (2) is based on the fact that summation of independent mean-zero $\sigma_i^2$-sub-Gaussian ($i = 1,\hdots,m$) is $\sum_{i=1}^m \sigma_i^2$-sub-Gaussian. (1) and (2) indicate that given $\x^k$, $\bar w^k_j$ is $\sigma_\infty^2/m$-sub-Gaussian, for any $j = 1,\hdots,d$. Therefore, according to the property of sub-Gaussian, we have that $\bE[ \exp((\bar w^k_j)^2/(6\sigma_\infty^2/m)) \mid \x^k] \le 2$. Then, for any $t > 0$,
    \begin{align*}
       & \exp(t \bE[ \| \bar w^k \|_\infty^2 \mid \x^k]) \\
       & \le \bE[\exp( t \max_{1 \le j \le d} (\bar w^k_j)^2 )\mid \x^k] \\
       & = \bE[ \max_{1 \le j \le d} \exp( t (\bar w^k_j)^2 ) \mid \x^k] \\
       & \le \bE[ \sum_{j = 1}^d \exp( t (\bar w^k_j)^2 ) \mid \x^k] \\
       & = \sum_{j = 1}^d \bE[ \exp( t (\bar w^k_j)^2 )\mid \x^k ],
    \end{align*}
    where the first inequality is from Jensen's inequality. If we let $t = m/(6\sigma_\infty^2)$, then
    \begin{align*}
        & \exp\left( \frac{m}{6\sigma_\infty^2} \bE[ \| \bar w^k \|_\infty^2 \mid \x^k] \right) \le 2d \implies \bE[ \| \bar w^k \|_\infty^2 \mid \x^k] \le 6\log(2d) \sigma_\infty^2/m,
    \end{align*}
    and we proved (3).
    \qed
\end{proof}
Proof of Lemma~\ref{lm: svrgbd}.
\begin{proof}
If we let $$
B^k(\,\cdot\,):=\frac{1}{m_k}  \sum_{i=1}^{m_k} \nabla f(\,\cdot\,,\zeta^k_i),
$$
we then have for any $k$,
\begin{align}\label{ineq: vr}
\begin{aligned}
    & \E [ \| \nabla f(\x^k) - B^k(\x^k) + B^k(\x^{n_k}) - \nabla f(\x^{n_k})  \|^2 \mid \x^k,\x^{n_k} ] \\
    & = \E [ \| \nabla f(\x^k) - \nabla f(\x^k, \zeta^k_1) + \nabla 
f(\x^{n_k}, \zeta^k_1) - \nabla f(\x^{n_k})  \|^2 \mid \x^k,\x^{n_k} ]/m_k,
\end{aligned}
\end{align}
since $\nabla f(\x^k, \zeta^k_i) - \nabla f(\x^{n_k}, \zeta^k_i)$, $i = 1,\hdots,m_k$ are i.i.d. with mean $\nabla f(\x^k) - \nabla f(\x^{n_k})$ given $\x^k$ and $\x^{n_k}$.
Therefore,
\begin{align}
    \notag
    & \bE[ \| \nabla f(\xbf^k) - G^k \|_\infty^2 ] \\
    \notag
  &  =  \,\bE[ \| \nabla f(\xbf^k) - B^k(\xbf^k) + B^k(\xbf^{n_k}) - G^{n_k} \|_\infty^2 ] \\
    \notag
  &  = \,\bE[ \| \nabla f(\xbf^k) - B^k(\xbf^k) + B^k(\xbf^{n_k}) - G^{n_k} - \nabla f(\xbf^{n_k}) + \nabla f(\xbf^{n_k}) \|_\infty^2 ] \\
    \notag
   &   \le\, 2 \bE[ \| \nabla f(\xbf^k) - B^k(\xbf^k) + B^k(\xbf^{n_k}) - \nabla f(\xbf^{n_k}) \|^2 ] + 2 \bE[ \| - G^{n_k} + \nabla f(\xbf^{n_k}) \|_\infty^2 ] \\
    \notag
   &  \overset{\eqref{redvar-SVRG1},\eqref{ineq: vr}}{\le}  \,2\bE[ \bE [  \| \nabla f(\xbf^k) - \nabla f(\xbf^{n_k}) - \nabla f(\xbf^k,\zeta^k_1) + \nabla f (\xbf^{n_k},\zeta^k_1) \|^2 \mid \xbf^k, \xbf^{n_k}] / m_k ] \\
   \notag
   & + \frac{2c(\log d)\sigma_\infty^2}{m_{n_k}} \\
    \notag
   &  \le  \,
   \frac{4}{m_k} \bE[  \bE [  \| \nabla f(\xbf^k) - \nabla f(\xbf^{n_k}) \|^2 + \| - \nabla f (\xbf^k,\zeta^k_1) + \nabla f(\xbf^{n_k},\zeta^k_1) \|^2 \mid \xbf^k, \xbf^{n_k}]] \\
   \notag
   & + \frac{2c(\log d)\sigma_\infty^2}{m_{n_k}} \\
    \notag
   &  =  \,
   \frac{4}{m_k} \bE[ \| \nabla f(\xbf^k) - \nabla f(\xbf^{n_k}) \|^2 ]
    + 4 \bE[ \bE[ \| - \nabla f (\xbf^k,\zeta^k_1) + \nabla f (\xbf^{n_k},\zeta^k_1) \|^2 \mid \xbf^k, \xbf^{n_k}] / m_k ] \\
    \notag
   & + \frac{2c(\log d)\sigma_\infty^2}{m_{n_k}} \\
    \notag
   &  \le  \, \frac{4L^2}{m_k} \bE[ \| \xbf^k - \xbf^{n_k} \|^2 ] + 4 \bE[ \bE [L^2 \| \xbf^k - \xbf^{n_k} \|^2 \mid \x^k, \x^{n_k} ] / m_k ] + \frac{2c(\log d)\sigma_\infty^2}{m_{n_k}} \\
   \label{ineq: redvar}
   &  =  \,\frac{8 L^2}{m_k} \bE[ \| \xbf^k - \xbf^{n_k} \|^2 ] + \frac{2c(\log d)\sigma_\infty^2}{m_{n_k}}.
\end{align}
Note that for $t \ge n_k$,
\begin{align}
\notag
    \| \xbf^t - \xbf^{n_k} \|^2 & = \left\| \sum_{i = n_k}^{t-1} (\xbf^{i+1} - \xbf^i) \right\|^2 = (t - n_k)^2 \left\| \frac{1}{t-n_k} \sum_{i = n_k}^{t-1} (\xbf^{i+1} - \xbf^i) \right\|^2 \\
    \label{ineq: nn1}
    & \le (t - n_k)^2 \cdot \frac{1}{t-n_k} \sum_{i = n_k}^{t-1} \| \xbf^{i+1} - \xbf^i \|^2 = (t - n_k) \sum_{i = n_k}^{t-1} \| \xbf^{i+1} - \xbf^i \|^2.
\end{align}
Therefore, for any $k$,
\begin{align}
\notag
    \sum_{t = n_k}^k \bE[ \| \xbf^t - \xbf^{n_k} \|^2 ] & \overset{\eqref{ineq: nn1}}{\le} \sum_{t = n_k}^k (t - n_k) \sum_{i = n_k}^{t-1} \bE[ \| \xbf^{i+1} - \xbf^i \|^2 ] \\
    \notag
    & \le \sum_{t = n_k}^{k-1} \frac{(1+k -n_k)(k-n_k)}{2} \bE[\| \xbf^{t+1} - \xbf^t \|^2] \\
    \notag
    & \le q^2 \sum_{t = n_k}^{k-1} \bE[\| \xbf^{t+1} - \xbf^t \|^2] \\
    \label{ineq: nn2}
    & \le q^2 \sum_{t = n_k}^k \bE[\| \xbf^{t+1} - \xbf^t \|^2].
\end{align}
Note that we assume $m_k \equiv m$ if $k \neq n_k$ and $m_{n_k} \equiv m_1$. Therefore,
\begin{align}
\notag
    & \bE[ \| \nabla f(\x^Y) - G^Y \|_\infty^2 ] \\
    \notag
    & = \frac{1}{K} \sum_{k=1}^K \bE[ \| \nabla f(\x^k) - G^k \|_\infty^2 ] \\
    \notag
    & \overset{\eqref{ineq: redvar}}{\le} \frac{8L^2}{Km} \sum_{k=1}^K \bE[\| \x^k - \x^{n_k} \|^2] +\frac{2c(\log d)\sigma_\infty^2}{m_1} \\
    \notag
    & \overset{\eqref{ineq: nn2}}{\le} \frac{8L^2q^2}{Km} \sum_{k=1}^K \bE[\| \x^{k+1} - \x^k \|^2] +  \frac{2c(\log d)\sigma_\infty^2}{m_1} \\
    \notag
    & = \frac{8L^2q^2}{m} \bE[ \| \x^{Y+1} - \x^Y \|^2 ] + \frac{2c(\log d)\sigma_\infty^2}{m_1}.
\end{align}\qed
\end{proof}

Proof of Theorem~\ref{thm: case1svrg}.
\begin{proof}
Note that similar to deriving \eqref{errbd3}, \eqref{errbd1}\eqref{setting1}\eqref{subgbd}\eqref{ineq: phi} lead to 
\begin{align}
\notag
& \left( \frac{1}{\eta} - \frac{L}{2} - \frac{\gamma}{2 \eta} \right) \bE [ \| \Delta \xbf^{Y+1} \|^2 ] + \frac{\hrho}{8\eta} \bE[ \| \ybf^{Y+1} - \xbf^Y \|_1^2 ] \\\notag
& \le \frac{\Delta}{K} +  \frac{\tau + t}{2}  \bE[ \| g^Y - G^Y \|_\infty^2 ] + \left( \frac{1}{2\tau} + \frac{1}{2\eta\gamma} + \frac{2\hrho}{\eta} \right) \eh^2 \\\notag
& \overset{\eqref{redvar-SVRG2}}{\le} \frac{\Delta}{K} +  \frac{\tau + t}{2}  \left( \frac{8L^2q^2}{m} \bE[ \| \x^{Y+1} - \x^Y \|^2 ] + \frac{2c(\log d)\sigma_\infty^2}{m_1} \right) + \left( \frac{1}{2\tau} + \frac{1}{2\eta\gamma} + \frac{2\hrho}{\eta} \right) \eh^2 \\\notag
& \implies \left( \frac{1}{\eta} - \frac{L}{2} - \frac{\gamma}{2 \eta}  - \frac{4(\tau + t)L^2q^2}{m} \right) \bE [ \| \Delta \xbf^{Y+1} \|^2 ] + \frac{\hrho}{8\eta} \bE[ \| \ybf^{Y+1} - \xbf^Y \|_1^2 ] \\
\label{errbd7}
& \le \frac{\Delta}{K} + \frac{(\tau + t)c(\log d)\sigmainf^2}{m_1} + \left( \frac{1}{2\tau} + \frac{1}{2\eta\gamma} + \frac{2\hrho}{\eta} \right) \eh^2.
\end{align}
By \eqref{redvar-SVRG2}\eqref{errbd2}\eqref{subgbd}\eqref{errbd7},
\begin{align}
\notag
    & \bE[ {\rm dist}_{\| \cdot \|_\infty}(0, \partial(f+\delta_X)(\x^{Y+1})) ] \\
\notag
& \le \left(L + \frac{1}{\eta} \right) \sqrt{\bE[ \| \Delta \x^{Y+1} \|^2 ]} +  \sqrt{\bE[ \| g^Y - G^Y \|_\infty^2 ]} + \frac{1}{\eta} \bE[ \| \xi^{Y+1} \|_\infty ] + \frac{ \eh}{\eta} \\ \notag
& \le \left(L + \frac{1}{\eta} \right) \sqrt{\bE[ \| \Delta \x^{Y+1} \|^2 ]} + \sqrt{ \frac{8L^2q^2}{m} \bE[ \| \x^{Y+1} - \x^Y \|^2 ] + \frac{2c(\log d)\sigma_\infty^2}{m_1} }\\ \notag
& + \frac{ \hrho}{\eta} \bE[ \| \y^{Y+1} - \x^Y \|_1 ] + \frac{ \eh}{\eta} \\ \notag
& \le \left(L + \frac{1}{\eta} + \sqrt{\frac{8L^2q^2}{m}} \right) \sqrt{\bE[ \| \Delta \x^{Y+1} \|^2 ]} +  \sqrt{ \frac{2c(\log d)\sigma_\infty^2}{m_1} } + \frac{\hrho}{\eta} \bE[ \| \y^{Y+1} - \x^Y \|_1 ] + \frac{\eh}{\eta} \\ \notag
    & \le \left( L + \frac{1}{\eta} + R\sqrt{\frac{8L^2q^2}{m}} \right) \sqrt{\frac{\frac{\Delta}{K} + \frac{(\tau + t)c(\log d)\sigmainf^2}{m_1} + \left( \frac{1}{2\tau} + \frac{1}{2\eta\gamma} + \frac{2\hrho}{\eta} \right) \eh^2}{\frac{1}{\eta} - \frac{L}{2} - \frac{\gamma}{2\eta} - \frac{4(\tau + t)L^2q^2}{m} }}  \\ 
\notag
    & + \sqrt{ \frac{2c(\log d)\sigma_\infty^2}{m_1} } + \frac{\hrho}{\eta}\sqrt{\frac{\frac{\Delta}{K} + \frac{(\tau + t)c(\log d)\sigmainf^2}{m_1} + \left( \frac{1}{2\tau} + \frac{1}{2\eta\gamma} + \frac{2\hrho}{\eta} \right) \eh^2}{\hrho/(8\eta)}} + \frac{\eh}{\eta}.
\end{align}
\qed
\end{proof}
Proof of Theorem~\ref{thm: case2svrg}.
\begin{proof}
Similar to the derivation of \eqref{errbd5}, we have
\begin{align}
\notag
&  \left( \frac{1}{\eta} - \frac{L}{2} - \frac{\gamma}{2 \eta} \right) \bE [ \| \Delta \xbf^{Y+1} \|^2 ] + \frac{\psi - \eh}{\eta} \bE[ \hrho_Y ] \\
\notag
& \le \frac{\Delta}{K} +  \frac{\tau + t}{2}  \bE[ \| g^Y - G^Y \|_\infty^2 ] + \left( \frac{1}{2\tau} + \frac{1}{2\eta\gamma} \right) \eh^2 + \frac{\psi^2}{2t} \\
\notag
& \overset{\eqref{redvar-SVRG2}}{\le} \frac{\Delta}{K} + \frac{\tau+t}{2} \left( \frac{8L^2q^2}{m} \bE[ \| \x^{Y+1} - \x^Y \|^2 ] + \frac{2c(\log d)\sigma_\infty^2}{m_1} \right) + \left( \frac{1}{2\tau} + \frac{1}{2\eta\gamma} \right) \eh^2 + \frac{\psi^2}{2t} \\
\notag
\implies & \left( \frac{1}{\eta} - \frac{L}{2} - \frac{\gamma}{2 \eta} - \frac{4(\tau + t)L^2q^2}{m} \right) \bE [ \| \Delta \xbf^{Y+1} \|^2 ] + \frac{\psi - \eh}{\eta} \bE[ \hrho_Y ] \\
\label{errbd9}
& \le \frac{\Delta}{K} + \frac{(\tau+t)c(\log d)\sigmainf^2}{m_1} + \left( \frac{1}{2\tau} + \frac{1}{2\eta\gamma} \right) \eh^2 + \frac{\psi^2}{2t}
\end{align}
Then by \eqref{redvar-SVRG2}\eqref{errbd2},
\begin{align}
\notag
    & \bE[ {\rm dist}_{\| \cdot \|_\infty}(0, \partial(f+\delta_X)(\x^{Y+1})) ] \\
\notag
& \le \left(L + \frac{1}{\eta} \right) \sqrt{\bE[ \| \Delta \x^{Y+1} \|^2 ]} +  \sqrt{\bE[ \| g^Y - G^Y \|_\infty^2 ]} + \frac{1}{\eta} \bE[ \| \xi^{Y+1} \|_\infty ] + \frac{ \eh}{\eta} \\ \notag
& \le \left(L + \frac{1}{\eta} \right) \sqrt{\bE[ \| \Delta \x^{Y+1} \|^2 ]} + \sqrt{ \frac{8L^2q^2}{m} \bE[ \| \x^{Y+1} - \x^Y \|^2 ] + \frac{2c(\log d)\sigma_\infty^2}{m_1} } + \frac{1}{\eta} \bE[ \| \xi^{Y+1} \|_\infty ] \\
\notag
& + \frac{ \eh}{\eta} \\ \notag
& \overset{\eqref{subgbd2}}{\le} \left(L + \frac{1}{\eta} + \sqrt{\frac{8L^2q^2}{m}} \right) \sqrt{\bE[ \| \Delta \x^{Y+1} \|^2 ]} +  \sqrt{ \frac{2c(\log d)\sigma_\infty^2}{m_1} } + \frac{1}{\eta} \bE[  \hrho _Y ] + \frac{\eh}{\eta} \\ \notag
    & \overset{\eqref{errbd9}}{\le} \left( L + \frac{1}{\eta} + \sqrt{\frac{8L^2q^2}{m}} \right) \sqrt{\frac{\frac{\Delta}{K} + \frac{(\tau + t)c(\log d)\sigmainf^2}{m_1} + \left( \frac{1}{2\tau} + \frac{1}{2\eta\gamma} \right) \eh^2 + \frac{\psi^2}{2t} }{\frac{1}{\eta} - \frac{L}{2} - \frac{\gamma}{2\eta} - \frac{4(\tau + t)L^2q^2}{m} }}  \\ 
    \notag
    & + \sqrt{ \frac{2c(\log d)\sigma_\infty^2}{m_1} } + \frac{1}{\psi - \eh} \left( \frac{\Delta}{K} + \frac{(\tau + t)c(\log d)\sigmainf^2}{m_1} + \left( \frac{1}{2\tau} + \frac{1}{2\eta\gamma} \right) \eh^2 + \frac{\psi^2}{2t} \right) + \frac{\eh}{\eta}.
\end{align}\qed
\end{proof}

\section{Algorithm}\label{app: alg}
Here we provide the details of the related algorithms: Alternating direction method of multiplier (ADMM) to solve \eqref{proj}\eqref{def: proj}, gradient descent with backtracking (Algorithm \ref{alg: GD backtracking}) used for computing $f^*$ in the numerical experiment. 

{We start with the following structured convex program:
    \begin{align}\label{scp}
        \min \quad  \tf(w) + \tg(z) \quad \mbox{s.t.} \quad Aw + Bz = b, \; w\in \mathcal{X}, \; z \in \mathcal{Z},
    \end{align}
    where $\tf: \bR^n \to \bR$, $\tg: \bR^m \to \bR$ are convex functions and $\mathcal{X},\mathcal{Z}$ are closed and convex sets. The ADMM algorithm is as follows. 
    \begin{algorithm}[htbp] 
\caption{Alternating direction method of multiplier (ADMM) to solve \eqref{scp}}\label{alg: admm}
\begin{description}
\item[{\bf Step 1.}] Initialize primal variable $w^1$, dual variable $\lambda^1$, penalty parameter $\rho > 0$
\item[{\bf Step 2.}] Invoke the following operations for $t=1,2,...$
\begin{description}
 \item[{\bf Step 2.1}] $ w^{t+1} = \mbox{arg}\min\limits_{w \in \mathcal{X}} \tf(w) - w^T A^T \lambda^t + \frac{\rho}{2} \| A w + B z^t - b \|^2  $
 \item[{\bf Step 2.2}] $ z^{t+1} = \mbox{arg}\min\limits_{z \in \mathcal{Z}} \tg(z) - z^T B^T \lambda^t + \frac{\rho}{2} \| A w^{t+1} + B z - b \|^2  $
 \item[{\bf Step 2.3}] $\lambda^{t+1} = \lambda^t - \rho ( A w^{t+1} + B z^{t+1} - b ) $.
\end{description}
\end{description}
\end{algorithm} 
\paragraph{Linear convergence and calculation of $P_X^k(\cdot)$.} 
Under the assumption that $\partial \tf$ and $\partial \tg$ are piecewise linear multifunctions (recall that $F$ is a piecewise linear multifunction if $Gr(F) := \{ (x,y) \mid y \in F(x) \}$ is the union of finitely many polyhedra) and $\mathcal{X},\mathcal{Z}$ are polyhedra, \cite{yang2016linear} showed that ADMM enjoys linear convergence.\\
More specifically, denote $v^t = (z^t,\lambda^t) $, $v^* = (z^*, \lambda^*)$, $H = \left( \begin{array}{cc}
       \rho^2 B^T B &  \\
         & I
    \end{array}
    \right)$,
    \begin{align*}
    \Omega^* = \{ (z^*,\lambda^*) \mid \exists w^* \mbox{ such that $(w^*,z^*,\lambda^*)$ is the KKT solution of \eqref{scp}}\}. 
\end{align*}
Then according to \cite[Theorem 14]{yang2016linear},
\begin{align}\label{app:thm14}
    {\rm dist}_H^2(v^{t+1},\Omega^*) \le \tau {\rm dist}_H^2(v^t, \Omega^*),
\end{align}
where ${\rm dist}_H(\cdot,\Omega^*) \triangleq \inf_{z \in \Omega^*} \| \cdot - z \|_H$ and $\tau$ is an constant in $(0,1)$. Also, according to \cite[Theorem 8]{yang2016linear}, we have
\begin{align}\label{app:thm8}
    \| v^{t+1} - v^t \|_H^2 \le \| v^t - v^* \|_H^2 - \| v^{t+1} - v^* \|_H^2
\end{align}
for any $v^* \in \Omega^*$. Combining \eqref{app:thm14} and \eqref{app:thm8}, we have that $\| v^{t+1} - v^t \|_H^2$ converges to $0$ at a linear rate. To calculate $P_X^k(\cdot)$ defined in \eqref{def: proj} at $k$th iteration, we can invoke Algorithm~\ref{alg: admm} by assigning
\begin{align*}
    \tf(w) = \frac{1}{2}\| w - (\x^k - \eta G^k ) \|^2, \quad \tg(z) = \phi(z - \x^k), \quad \mathcal{X} = X, \quad \mathcal{Y} = \bR^d, \\
    A = I, \quad B = -I, \quad b = 0.
\end{align*}
Then each step of ADMM is easy to compute by our assumption. Moreover, $H= \left( \begin{array}{cc}
       \rho^2 I_d &  \\
         & I_d
    \end{array}
    \right)$ 
and the quantity $ \| v^{t+1} - v^t \|_H^2 = \rho^2 \| z^{t+1} - z^t \|^2 + \| \lambda^{t+1} - \lambda^t \|^2$ converges to $0$ in a linear rate. We stop the algorithm when 
    \begin{align}\label{stopcr}
    \rho \| z^{t+1} - z^t  \| \le \eh, \quad \| \lambda^{t+1} - \lambda^t \|_1/\rho \le \eh,
    \end{align}
    and let
    $$
    \x^{k+1} := w^{t+1},\quad \y^{k+1} := z^{t+1},\quad \xi^{k+1} := -\lambda^{t+1}, \quad \Gamma^{k+1} := \rho (z^{t+1} - z^t).
    $$
Then by the optimality conditions of Step 2.1 and Step 2.2 in Algorithm~\ref{alg: admm}, we have that \eqref{ineq: proj1} and \eqref{eq: proj2} hold. By \eqref{stopcr}, we have that \eqref{ineq: proj3} holds.
}

Next is the gradient descent with backtracking used in the numerical experiment.
\begin{algorithm} 
\caption{Projected gradient method with backtracking}\label{alg: GD backtracking}
\begin{description}
\item[{\bf Initialization Step}] Set $c_1 = \dfrac{1}{4}, \beta = \dfrac{1}{2}, \epsilon = 1e-10$. Initialize $\xbf^0 = \bm{0}$.
\item[{\bf Main Step}] Set $\alpha_k \equiv \alpha = 1$;\\
{\bf While} {$f(P_X (\xbf^k - \alpha_k \nabla f(\xbf^k))) > f(\xbf^{k}) + c_1 (\nabla f(\xbf^{k}))^T (P_X (\x^k - \alpha_k \nabla f(\x^k)) - \x^k) $}\\
$\alpha_k = \beta\alpha_k$
\item[{\bf Stopping}] Stop if $\lVert \xbf^{k+1} - \xbf^{k} \rVert_1 \le \epsilon$ and output $\xbf^{k}$; otherwise replace $k$ by $k+1$ and repeat the Main Step.
\end{description}
\end{algorithm}

\end{document}